\documentclass[a4paper]{article}
\usepackage[utf8x]{inputenc}
\usepackage{yfonts}
\usepackage{amsmath}
\usepackage{amsthm}
\usepackage{amssymb}
\usepackage{amsfonts}
\usepackage{paralist}
\usepackage{hyperref}
\usepackage[arrow, matrix, curve]{xy}
\usepackage{youngtab}
\usepackage{young}
\usepackage{graphicx}
\usepackage{wrapfig}
\usepackage{tikz}
\usepackage{mathtools}

\usetikzlibrary{calc}

\newtheorem{thm}[subsection]{Theorem}
\newtheorem{facts}[subsection]{Facts}

\newtheorem{lem}[subsection]{Lemma}

\newtheorem{cor}[subsection]{Corollary}
\newtheorem{defi}[subsection]{Definition}
\newtheorem{ex}[subsection]{Example}
\newtheorem{rem}[subsection]{Remark}

\newtheorem{prop}[subsection]{Proposition}

\newtheorem{notation}[subsection]{Notation}

\newtheorem{obs}[subsection]{Observation}

\newcommand{\A}{\mathbb{A}}

\newcommand{\F}{\mathbb{F}}

\newcommand{\C}{\mathbb{C}}
\newcommand{\E}{\mathbb{E}}

\newcommand{\R}{\mathbb{R}}
\renewcommand{\P}{\mathcal{P}}
\newcommand{\Q}{\mathbb{Q}}
\newcommand{\Z}{\mathbb{Z}}

\renewcommand{\cal}{\mathcal}

\newcommand{\N}{\mathbb{N}}
\newcommand{\K}{\mathbb{K}}

\renewcommand{\O}{\mathbb{O}}

\renewcommand{\-}{\text{-}}

\renewcommand{\l}{\lambda}

\newcommand{\q}{\textbf{q}}

\newcommand{\rar}{\rightarrow}
\newcommand{\lar}{\leftarrow}
\newcommand{\inj}{\hookrightarrow}
\newcommand{\linj}{\hookleftarrow}
\newcommand{\surj}{\twoheadrightarrow}

\newcommand{\ol}{\overline}

\newcommand{\wt}{\widetilde}

\begin{document}
\title{Modular equivariant formality}
\author{Jan Weidner}
\date{}
\maketitle
\begin{abstract}
Let $X$ be a partial flag variety, equipped with the Borel action by multiplication.
We give a criterion for the equivariant derived category with modular coefficients to be formal. 
\end{abstract}

\section*{Introduction}
Let $X$ be a space equipped with the action of a group $G$ and $R$ be a nice commutative ring. Then 
one can form the equivariant derived category \cite{BernsteinLunts}:
$$D^b_G(X,R)$$
It is the correct replacement for the naive derived category of the quotient $X/G$. 
In fact it exists even if $X/G$ does not or is badly behaved.

Our goal is to provide a description of this category. More precisely we want to construct equivalences 
of categories 
\begin{equation}\label{equivFormality}
D^b_G(X,R)\cong per\-Ext^\bullet(IC)
\end{equation}
where the right hand side is the perfect derived category of the $dg$-algebra $(Ext_{D^b_G(X,R)}^\bullet(IC,IC),d=0)$
and $IC$ denotes the direct sum of all intersection cohomology complexes of all $G$-orbit closures. Here we assume for simplicity that there are 
only finitely many orbits, each of which supports only finitely many irreducible local systems. 
If there is such an equivalence as in \ref{equivFormality}, we will say that the equivariant derived category is formal.
In fact equivariant formality is already known in many situations. The following table gives an (very incomplete and biased) 
overview:
$$
\begin{tabular}{l | l | l | l}
 $G$ & $X$ & $R$ & \text{ Reference} \\
\hline
 \text{connected Lie group} & \text{point} & $\R$ & \cite{BernsteinLunts} \\
\hline
  \text{Torus } & \text{toric variety } & $\R$ & \cite{LuntsToric}\\
\hline 
 \text{Borel } & \text{partial flag variety } & $\C$ & \cite{OlafThesis} \\
\hline 
  \text{complex semisimple adjoint} & \text{smooth complete symmetric} & $\C$ & \cite{SymSpaceFormal}\\ 
\hline
\end{tabular}
$$
Observe that all of these results require $R$ to be a field of characteristic zero.\footnote{See however \cite{RSW} for non-equivariant 
formality results away from characteristic zero.} The reason is, that the known techniques either 
use mixed sheaves or commutative models of cochain algebras, both of which are problematic without this assumption.

In this article, we would like to add some examples, where $R$ is of mixed or positive characteristic.
More precisely we are interested in the following situation: 
Let $B\subset P \subset G$ a complex connected reductive group, along with a Borel and a parabolic subgroup. 
Let $X=G/P$ be the corresponding partial flag variety, equipped with the $B$ action by multiplication on the left.
Then our goal is to prove the following theorem:
\begin{thm}
Suppose that all stalks and costalks of all $B$-constructible $\mathbb Z_l$-intersection 
cohomology complexes $IC_w^{\Z_l}$ on $X=G/P$ are torsion free. 
If $l>wr(X,B)$\footnote{This is a mild and explicit condition on $l$ to be explained later \ref{defBGSequiv}.}
then there exists an equivalence of categories:
$$D_B^b(X,\F_l)\cong per\-Ext^\bullet(IC)$$
\end{thm}
A typical case, where all assumptions can be checked is the Borel action on a Grassmannian:
\begin{cor}
 Suppose that $1,q,\ldots, q^{n+min(n-k,k)}$ are pairwise different elements of $\F_l$. Then there is an equivalence of categories:
$$D^b_B(Gr(k,n),\F_l)\cong per\-Ext^\bullet(IC)$$
\end{cor}

\subsection*{Outline}
Recall Schn\"urer's \cite{OlafThesis} formality result for the equivariant derived category of a partial flag variety:
$$D^b_B(G/P,\C) \cong per\-Ext^\bullet(IC)$$
His proof is based on a purity argument, carried out in the framework of mixed Hodge modules.
Substantial parts of this article are inspired by Schn\"urer's techniques. So let us first sketch a variant of his method in the case of 
$\Q_l$-coefficients.

There are standard techniques to pass between objects (varieties, sheaves, etc.) over $\C$ and 
their analogues over $\ol {\F}_q$. In order to exploit the formalism of weights, it is advantageous to work in the 
latter setting. Hence we work with the $\ol {\F}_q$-incarnations of $G,P,B,\ldots$ from now on.
After this preliminary translation, the argument can be carried out in three steps:
\begin{enumerate}
\item The starting point is the following reformulation of \cite[4.4.4]{BGS}:
\begin{thm}\label{PropPBGSkoszul2}
 Let $X=\bigsqcup_{\l\in \Lambda} X_\l$ be a cell stratified variety over $\ol \F_q$, subject to the BGS-condition\footnote{This condition is for example satisfied by all partial flag varieties.}. Then 
there exists a $\Q_l$-Koszul-algebra $A$ and an equivalence between perverse sheaves constant along cells and finitely generated $A$-modules:
$$\cal P_\Lambda(X,\Q_l)\cong mod\-A$$
\begin{proof}
\cite{janGrass}   
\end{proof}
\end{thm}
Now there is a well known paradigm that Koszulity and formality are closely related. 

More precisely let $A$ be a noetherian Koszul algebra and $L:=A_0$ be the direct sum of irreducible modules. 
Let $L^\bullet$ be a projective resolution and $End^\bullet(L^\bullet)$ be the endomorphism complex.
Then Morita theory gives us an equivalence (if $A$ has finite cohomological dimension):
$$D^b(mod\-A) = per \-A \cong per\-End^\bullet(L^\bullet)$$
Now Koszulity comes into play. We choose $L^\bullet$ to be a complex of graded projectives,
which induces a second grading on $End^\bullet(L^\bullet)$. This allows to apply Deligne's argument 
and find a roof of multiplicative quasi-isomorphisms connecting $End^\bullet(L^\bullet)$ and its cohomology
$Ext^\bullet(L)$.
Putting everything together we obtain:
\begin{align*}
 D^b_\Lambda(X,\Q_l)&=D^b(\cal P_\Lambda(X,\Q_l)) \\
 &=D^b(mod\-A) \\
 & \cong per \-End^\bullet(L^\bullet) \\
 & \cong per \-Ext^\bullet(L) \\
 &=per \-Ext^\bullet(IC)
\end{align*}
Here we also used, that the realization functor $D^b(\cal P_\Lambda(X,\Q_l))\rar  D^b_\Lambda(X,\Q_l)$ is an equivalence \cite[2.3.4]{RSW}.
\item Now by \cite[Prop 95]{OlafThesis} one can write the equivariant derived category as a projective limit of non-equivariant derived categories:
$$D^b_{G}(X,\Q_l)=\varprojlim D^b_\Lambda(\ol X_n,\Q_l) $$
Here the $\ol X_n$ are approximations of the quotient $X/G$, which does usually not exist is the category of varieties.
A way to construct suitable $\ol X_n$ is to take $\ol X_n:=X_n/G:=(X\times E_n)/G$, where $\varinjlim E_n=EG$.
For example if $X=pt$ and $G=\mathbb G_m$, then we could choose $\ol X_n=\mathbb P^n$.
Anyway we need to find approximations $X_n$, such that each $\ol X_n$ satisfies the assumptions of \cite[4.4.4]{BGS}. 
\item Finally abstract arguments \cite[Thm 44, Prop 86]{OlafThesis} allow to deduce formality in the limit:
$$D^b_{G}(X,\Q_l)=\varprojlim D^b_\Lambda(\ol X_n,\Q_l) =\varprojlim per \- Ext^\bullet(IC)  = per \- Ext^\bullet(IC) $$
where the first $Ext$-algebra is in $D^b_\Lambda(\ol X_n,\Q_l)$ and the second in $D^b_{G}(X,\Q_l)$.

\end{enumerate}
Let us analyze what problems occur, if we try to replace $\Q_l$ be $\Z_l$.
The main difficulty is of course that the description of perverse sheaves by a Koszul ring
\ref{PropPBGSkoszul2} is a priori not available. 
Hence we need a substitute. In order to formulate it, we recall some notation. For $X_0$ a 
variety satisfying the BGS condition, we denote by $wt(X)$ the set of Frobenius eigenvalues (up to a technical refining) on 
$End(P^{\Q_l})$, where $P$ is a minimal projective generator of $\P_\Lambda(X,\Q_l)$ equipped with a suitable lift to $X_0$.
The set $wt(X)$ consists of powers of $q$ and we say that it is separated if its cardinality stays the same when reducing modulo $l$.
 \begin{thm}\label{thmMainModKoszul2} 
 Let $X_0$ be a cell stratified variety. Assume that all $IC^{\Z_l}$-sheaves are parity and that
the BGS-condition  holds.
Let $P^{\Z_l}=\bigoplus P_\l$ be a minimal projective generator of $\cal P_\Lambda(X,\Z_l)$. If $wt(X)$
is separated, then
$$A:=End(P^{\Z_l})$$
admits a $\Z_l$-Koszul grading.\footnote{This is a version of Koszulity ``relative'' over $\Z_l$, see \cite{janGrass}. For example 
the polynomial ring $\Z_l[x_1,\ldots x_n]$ is $\Z_l$-Koszul and the exterior algebra is its dual.
}
\begin{proof}
\cite{janGrass}   
\end{proof}
\end{thm}
But now checking the assumptions for all $\ol X_w$ is a non trivial task. We need to make sure that:
\begin{itemize}
\item All $IC^{\Z_l}$-sheaves on each approximation step $\ol X_n$ are parity.
\item There exists a bound on the Frobenius eigenvalues occurring, which is uniform over all $X_n$.
\end{itemize}
There are two key insights for this task. 
\begin{enumerate}
\item Controlling parity and eigenvalues for $G$ acting on $X$ can be reduced to controlling 
parity and eigenvalues for $X$ (non-equivariant) and $G$ acting on a point.
\item For a connected solvable group (for example a Borel) acting on a point parity and eigenvalue questions 
boil down to easy questions about perverse sheaves on $\mathbb P^n$. 
\end{enumerate}
Let us be more precise about the first statement.
Let $X$ be an acyclically stratified variety, which satisfies the BGS-condition
\footnote{For example partial flag varieties satisfy this condition.}, equipped with a compatible action by a group $G$. 
Suppose that there exists an approximation $B_n$ of $BG$ 
that satisfies the assumptions of \ref{thmMainModKoszul2} with a uniform bound. Then we will construct an approximation $X_n$ 
subject to the above conditions. Hence $D^b_{G}(X,\Z_l)$ is formal. 
The strategy of the proof is as follows:  
Let $E_n\rar B_n$ be the universal bundle. Let us pretend for a moment, that this bundle is trivial
\footnote{Of course, quite the opposite is actually 
case.} 
$E_n=G\times B_n$. 

Then one can check, that a product inherits the following from its factors:
\begin{itemize}
\item The property of the $IC$-sheaves being parity.
\item The BGS-condition.
\item The bounds on relevant Frobenius eigenvalues are the sum of the bounds of the factors. 
\end{itemize}
Hence $X_n:=E_n\times X$ would give us the desired approximation.
Now it remains to fix the problem, that $E_n\rar B_n$ is not trivial.
For this purpose we will show, that all three of the above items can be controlled locally in the Zariski topology. 
Hence if $E_n\rar B_n$ is locally trivial in the Zariski topology everything is fine.

This leads us to the second question: Why is $E_n\rar B_n$ easy for solvable $G$?
For example if $G=\mathbb G_m$, we can choose $BG_n=\mathbb P^n$. 
If $G=T$ is a torus instead, $BT_n$ will be a product of projective spaces, which is still very manageable.
Finally any connected solvable group is homotopy equivalent to a torus, which again allows us to reduce to projective space.

\subsection*{Acknowledgments}
I want to thank Wolfgang Soergel, Matthias Wendt and Geordie Williamson for helpful discussions and interest.
This work was supported by the DFG via SPP 1388.

\section{The non-equivariant situation}
In this section we recall terminology and results in the non-equivariant situation. 

Let $\K$ be a finite extension of $\Q_l$ and denote by $\O$ its ring of integers. Let $\varpi \in \O$ be a 
uniformizing parameter and $\F:=\O/\varpi$ be
the residue field. For instance $\K=\Q_l$ and $\O=\Z_l$ and $\F=\F_l$.
Let $l \neq p$ be primes and $q$ be a power of $p$. 
\subsection{The six functors}

By $D^b_c(X,\E)$ we denote constructible derived category of a variety $X$ 
over a perfect field of characteristic different from $l$ with coefficients in $\E$.
It fits into a six functor formalism $f_*,f^*,f_!,f^!,\cal Hom,\otimes$\footnote{From now on, we will often use 
the same notation for a functor and its derived counterpart. For example $\otimes$ means $\overset{L}{\otimes}$ etc.}. 
The six functors commute with extension of scalars and pullback from varieties $X_0$ over $\F_q$ to their basechange $X$ over $\ol \F_q$. For example we have a canonical isomorphism
$$\F \otimes \cal Hom(\cal F, f_! \cal G)=\cal Hom(\F\otimes \cal F,  f_! (\F\otimes\cal G))$$
References for the six functor formalism are for example:
\begin{itemize}
\item \cite{SGA4} for $\E=\F$.
\item \cite{Ekedahl} for the passage from $\F$ to $\O$.
\item \cite{Weil2} for the passage from $\O$ to $\K$.
\end{itemize}

\subsection{Acyclically stratified varieties}
We recall and tweak some basic definitions and notation from \cite{RSW} and \cite{janGrass}.
Let $X$ be a variety over a field $k$, together with a finite decomposition into locally closed smooth affine irreducible subvarieties:
 $$X=\bigsqcup_{\l \in \Lambda} X_{\l}$$
We will denote the dimension of $X_\l$ by $d_\l$ and the inclusion by $l_\l:X_\l \inj X$. 
The inclusion of the closure of a stratum will be denoted by $\ol l_\l : \ol X_\l \inj X$.

If $k$ is algebraically closed, we say that $X=\bigsqcup_{\l \in \Lambda} X_\l$ is a stratification, if
$$l^*_\l l_{\mu*} \E$$
has constant cohomology sheaves for all $\l,\mu$. A cell stratification is a stratification, such that
$X_\l \cong \A^{d_\l}$. 
Typical examples of cell stratified varieties are partial flag varieties equipped with their decomposition into Bruhat cells:
$$G/P=\bigsqcup BwP/P$$

An acyclic stratification is a stratification all of whose strata are acyclic. In other words this means, that the strata's cohomology
looks like the cohomology of $\mathbb A^n$:
$$H^\bullet(X_\l,\O)=\O$$
Typical examples of acyclic stratifications arise by taking fiber bundle with fibers $\A^n$ over cell stratifications.
In literature results are often stated with cell stratification assumptions, but proofs only use acyclicity. 
We will cite such statements without further notification.

Given an acyclically stratified variety, we denote by $D^b_\Lambda(X,\E)$ the category of all constructible complexes $\cal F$ such that 
$l_\l^*\cal F$ and $l_\l^! \cal F$ both have constant cohomology sheaves for all $\l$.
It is an idempotent complete triangulated category. 
We will be sloppy and usually refer to objects of $D^b_\Lambda(X,\E)$ as sheaves. 
By $\P_\Lambda(X,\E)\subset D^b_\Lambda(X,\E)$
we denote the full subcategory of perverse sheaves.

To be more precise in the case $\E=\O$ there are two dual categories 
which one might call ``perverse sheaves''. We use the $t$-structure $p_{1/2}$ and not $p_{1/2}^+$ in the terminology of \cite[3.3.4.]{BBD}.
In other words, we choose the one which gives $\P(pt,\O)=mod- \O$. 

If $X=G/P$ is a partial flag variety,
equipped with the Bruhat stratification, we also use the notations 
$$D^b_{(B)}(X,\E):=D^b_\Lambda(X,\E) \text{ and } \P_{(B)}(X,\E):=\P_\Lambda(X,\E)$$

From now on we work with objects (varieties, sheaves, \ldots) defined over a field $k$, where $k$ is either $\F_q$ or its algebraic closure $\ol \F_q$.
As usual objects over $\F_q$ are denoted by $X_0,\cal F_0,\ldots$, while their 
base-change to $\ol \F_q$ is denoted by $X,\cal F,\ldots$.

We say that a locally closed decomposition $X_0=\bigsqcup_{\l \in \Lambda} X_{\l,0}$ is a (acyclic) stratification if its basechange
$X=\bigsqcup_{\l \in \Lambda} X_{\l}$ is. In this case
we denote by $D^b_\Lambda(X_0,\E)$ all constructible complexes whose base-change lands in $D^b_\Lambda(X,\E)$.
Again this is an idempotent complete triangulated category.
The category $\cal P_\Lambda(X_0,\E)\subset D^b_\Lambda(X_0,\E)$ is defined in a similar way.


\begin{notation}\label{notationCanonicalSheaves}
 Let $X$ be an acyclically stratified variety, and $X_\l$ be a stratum. Then there are a couple of canonically associated perverse sheaves
on $X$. We will introduce notation for them here:
\begin{itemize}
\item Let $\Delta_\l:=\Delta^{\E}_\l:={l_\l}_! \E[d_\l]$ denote the standard perverse sheaf.
\item Let $\nabla_\l:=\nabla^{\E}_\l:={l_\l}_* \E[d_\l]$ denote the costandard perverse sheaf.
\item Let $IC_\l:=IC^{\E}_\l:={l_\l}_{!*} \E[d_\l]$ denote the intersection cohomology complex. 
\end{itemize}
If the stratification is defined over $\F_q$, the same formulas define $\Delta_{\l,0},\nabla_{\l,0},IC_{\l,0}$.
\end{notation}

\subsection{Non-equivariant formality}
We recall the main result of \cite{janGrass}. 
It works only for spaces which satisfy two conditions: $IC^\O$-parity and the BGS-condition. 

Recall from \cite{JMWparity} that an object $\cal F\in D^b_\Lambda(X,\E)$ is called 
even (odd) if for all $\l$ the objects $l^*_\l(\cal F)$ and $l^!_\l(\cal F)$ 
have constant torsion free cohomology sheaves which vanish in odd (even) degrees.
An object $\cal F\in D^b_\Lambda(X,\E)$ is called parity, if it is a direct sum of an even and an odd object.
\begin{defi}
 Let $X=\bigsqcup X_\l$ be a stratified variety. We say that $X$ satisfies $IC^\E$-parity, if for each $\l$ the sheaf $IC^\E_\l$
is parity.
\end{defi}
\begin{defi}\label{condTate}
Let $X_0=\bigsqcup X_{\l,0}$ be a stratification. 
We say that it satisfies the BGS-condition or that it is BGS, if 
for all $i\in \Z$ and $\l, \mu \in \Lambda$ the sheaf $\cal H^i(l_\mu^* {l_\l}_{!*}\cal \K [d_\l])$
vanishes if $i+d_\l$ is odd and is isomorphic to a direct sum of copies of $\K(\frac{-d_\l-i}{2})$ if $i+d_\l$ is even.
\end{defi}
Given an acyclically stratified variety, say satisfying $IC^{\E}$-parity, we can extend our list \ref{notationCanonicalSheaves} of canonical sheaves
by projective covers $P_\l^\E \surj IC_\l^\E$. See \cite{janGrass} for their precise construction.
Each $P_\l$ admits also a lift $P_{\l,0}$ to $X_0$, which is constructed in \cite{janGrass}.
Let $P:= \bigoplus P_\l$. Then $P$ is a minimal projective generator and the lifts $P_{\l,0}$ induce a Frobenius action on $End(P)$. 
Up to a suitable normalization the eigenvalues of this action are recorded by the set $wt(X) \subset \{1, \q, \q^2, \ldots\}$.
We define $wr(X)\in \N$ to be greatest exponent occurring $wt(X)$ plus one. We say that $wt(X)$ is separated if $\q \mapsto q$ induces an injection
$wt(X) \inj \F$. See \cite{janGrass} for precise definitions.
We are now able to state the promised theorem:
\begin{thm}\label{thmMainModKoszul}
 Let $X_0$ be a cell stratified variety. Suppose that the $IC^\O$-sheaves are parity and that 
the BGS-condition holds. 
Let $P^\E=\bigoplus P_\l$. If $wt(X)$ is separated, then
$$A:=End(P^\E)$$
admits a $\E$-Koszul\footnote{This is a relative variant of the notion of Koszul grading. 
See \cite{janGrass} for a precise definition.} grading.
\begin{proof}
\cite{janGrass}
\end{proof}
\end{thm}
Formality is a consequence of Koszulity:
\begin{cor}
 Let $X_0$ be a cell stratified variety. Suppose that the $IC^\O$-sheaves are parity and that 
the BGS-condition holds. If $wt(X)$ is separated, then there is an equivalence of categories:
$$D^b_\Lambda(X,\E) \cong per \- Ext^\bullet(IC)$$
\begin{proof}
We argue that there is a chain of equivalences as follows:
$$D^b_\Lambda(X,\E) = D^b(\P_\Lambda(X,\E)) = D^b(mod \- A) \cong per \- End^\bullet(L^\bullet) = per \- Ext^\bullet(IC)$$
First the realization functor $D^b(\P_\Lambda(X,\E)) \rar D^b_\Lambda(X,\E)$ is an equivalence since all strata are acyclic
\cite[2.3.4]{RSW}.
Now by \ref{thmMainModKoszul} we know that $D^b(\P_\Lambda(X,\E)) = D^b(mod \- A)$ for a Koszul ring $A$. 
Let $L$ be the $A$-module corresponding to the direct sum of all $IC$-complexes and $L^\bullet$ by a resolution by graded projectives.
Then $Hom^\bullet(L^\bullet,\_)$ induced a fully faithful functor between our category and the derived category of the dg-algebra $End^\bullet(L^\bullet)$:
$$D^b(mod \- A)\rar Der \-End^\bullet(L^\bullet)$$
Under this functor the object $L\cong L^\bullet$ is mapped to $End^\bullet(L^\bullet)$. 
Since both categories are idempotent complete, our functor restricts to an equivalence 
$$D^b(mod \- A)= \langle L \rangle^{\ominus} \rar \langle End^\bullet(L^\bullet) \rangle^{\ominus} =per \-End^\bullet(L^\bullet)$$
Here $\langle T \rangle^{\ominus}$ denotes the thick subcategory generated by some object $T$.
Finally Koszulity allows us to construct a roof of quasi-isomorphisms connecting the $dg$-algebras via Deligne's argument
\cite[Prop 6]{OlafThesis}:
$$End^\bullet(L^\bullet) \lar S \rar Ext^\bullet(L^\bullet)=Ext^\bullet(IC)$$ 
\end{proof}
\end{cor}

\section{The equivariant derived category}
The equivariant derived category was introduced in \cite{BernsteinLunts}. In the following section we recall its definition and basic properties. 
While Bernstein and Lunts state their formalism in the topological setting, the essential ingredient 
is smooth basechange, which is also available for $l$-adic sheaves. As they remark, this allows to transfer the results 
to our situation. The translation is straightforward, after one has established a formalism of ``torsors''.

\subsection{$G$-torsors}
Let $X$ be a $G$-space. If the action is nice, the equivariant derived category is literally
the derived category of $X/G$. Otherwise is a well behaved replacement for the possibly non-existent category $D^b_c(X/G,\E)$.
To turn this idea into a precise definition, we should first explain what we mean by quotient.

Given a $G$-variety $X$ and a map $f:X\rar T$ into some other variety $T$, we say that $f$ is invariant, if 
$$f \circ \rho = f \circ \pi$$
where $\rho$ and $\pi$ are the action and the projection respectively. 
The quotient $X/G$ is the coequalizer of the action and projection map (if it exists):
$$X/G:=coker(G\times X \rightrightarrows X)$$
In other words it is an initial invariant map. Often it does not exist and even if it does it is badly behaved. 
It will be well behaved, if $X\rar X/G$ is a torsor:
\begin{defi}
 Let $G$ be an algebraic group. 
\begin{itemize}
 \item 
A $G$-torsor consists of a $fpqc$\footnote{In our situation, the word $fpqc$ can be replaced by $flat$ and surjective 
everywhere.}-morphism of varieties $p:\wt T \rar T$ 
and a $G$-action $\rho: G\times \wt T \rar \wt T$ such that $p$ is invariant under the action and in addition the diagram 
$$
\begin{xy}
 \xymatrix{
G \times \wt T \ar[d]_\pi \ar[r]^\rho & \wt T \ar[d]^p \\
\wt T \ar[r]^p & T
}
\end{xy}
$$
is cartesian.
\item
A map of torsors $p:\wt T \rar T$ to $q:\wt S \rar S$ consists of an equivariant map $\wt f:\wt T \rar \wt S$ and 
a map $f:S\rar T$ such that the diagram 
$$
\begin{xy}
 \xymatrix{
\wt T \ar[d] \ar[r] & \wt S \ar[d] \\
T \ar[r] & S
}
\end{xy}
$$
commutes.
\item A torsor is called trivial, if it is isomorphic to $G\times T\rar T$, where the $G$-action is by left multiplication on the first factor.
\item 
Given a topology\footnote{By topology we mean a Grothendieck pre-topology.} $\tau$ on the category of varieties, we say that a $G$-torsor is locally trivial with respect to $\tau$,
if there exists a $\tau$-covering $(U_i\rar T)_{i \in I}$ and maps of torsors 
$$
\begin{xy}
 \xymatrix{
\wt U_i  \ar[d] \ar[r] & \wt T \ar[d] \\
U_i \ar[r] & T
}
\end{xy}
$$
such that each $\wt U_i$ is trivial.
\end{itemize}
\end{defi}

By definition any torsor is locally trivial with respect to the $fpqc$-topology. 
Let us verify that the base of a torsor is the quotient of the total space by the action:
\begin{lem}\label{LemBaseTorsorIsQuot}
 Let $\pi:X\rar Y$ be a $G$-torsor.
 Then we have $Y=X/G$. 
\begin{proof}
We need to show that $Y=coker(G\times X \rightrightarrows X)$. In other words for every variety $T$ we need to give a natural identification:
$$Hom(Y,T) \overset{!}{=} ker(Hom(X,T)\rightrightarrows Hom(G\times X,T))$$
By \cite[2.55]{vistolyfib} the functor of points of any scheme $T$ is a sheaf in the $fpqc$-topology.
Applying this to the $fpqc$-cover $X\rar Y$ yields the result.
\end{proof}
\end{lem}

\begin{lem}
 Given a map of torsors, the diagram 
$$
\begin{xy}
 \xymatrix{
\wt T \ar[d] \ar[r]^{\wt f} & \wt S \ar[d] \\
T \ar[r]^f & S
}
\end{xy}
$$
is automatically cartesian.
\begin{proof}
A map of torsors $(f,\wt f)$ is the same thing as a map $f:T\rar S$ along with a map 
$\wt T \rar T\underset{S}{\times} \wt S$ over $T$ which is $G$-equivariant. Hence we may assume $S=T$ and $f=id$.
Using \ref{lemCartesianFpqcLocal} we may pull back the diagram along $\wt S\rar S$ and assume that
$\wt S\rar S=T$ is trivial.
Again by \ref{lemCartesianFpqcLocal} we may pull back along $\wt T\rar T$ and hence assume that $\wt T$ is trivial as well.
Now an inspection shows, that any endomorphism of the trivial torsor is an automorphism.
\end{proof}
\end{lem}

If we want to test, whether a diagram is cartesian, we can do so locally in the $fpqc$-topology:
\begin{lem}\label{lemCartesianFpqcLocal}
 Let $\wt X\rar X$ be a $fpqc$-map of varieties and $D$ be a commutative square:
$$
\begin{xy}
 \xymatrix{
Y' \ar[d] \ar[r] & X' \ar[d] \\
Y \ar[r]&  X
}
\end{xy}
$$
Let $\wt D$ be the pullback of $D$ along $\wt X\rar X$: \footnote{By this we mean, that we are given a commutative cube 
with back side $D$, front $\wt D$ and all other sides cartesian.}
$$
\begin{xy}
 \xymatrix{
\wt Y' \ar[d] \ar[r] & \wt X' \ar[d] \\
\wt Y \ar[r] & \wt X
}
\end{xy}
$$
If $\wt D$ is cartesian, then $D$ is also cartesian
\begin{proof}
Let $T$ be any variety. We need to find a canonical identification 
$$Hom(Y',T) \overset{!}{=} Hom(Y \underset{X}{\times} X',T)$$ Consider the diagram
$$
\begin{xy}
\xymatrix{
 & \wt Y' \ar[d] && \\
\wt Y' \underset{Y'}{\times} \wt Y' \ar[ru] \ar[d] & Y' \ar[dd] \ar[rr] && X' \ar[dd] \\
\wt Y' \ar[ru] \ar[dd] \ar[rr] && \wt X' \ar[dd] \ar[ru] & \\
& Y \ar[rr] && X \\
\wt Y \ar[rr] \ar[ru] && \wt X \ar[ru] & 
}
\end{xy}
$$
where all rectangles are cartesian, except possibly $D$. 
Using faithfully flat descent we have 

\begin{equation}
 Hom(Y',T)=\ker(Hom(\wt Y',T) \rightrightarrows Hom(\wt Y' \underset{Y'}{\times} \wt Y', T)) 
\end{equation}

Now we observe, that 
$$
\begin{xy}
 \xymatrix{
& \wt Y'\\
\wt Y' \underset{Y'}{\times} \wt Y' \ar[d] \ar[ru] & \\
\wt Y' & 
}
\end{xy}
$$
does not depend on $Y'$ at all! Indeed one computes
\begin{align*}
 \wt Y'&=\wt Y \underset{\wt X}{\times} \wt X' \\
	&=(Y\underset{X}{\times} \wt X) \underset{\wt X}{\times} (\wt X \underset{X}{\times} X') \\
&=Y\underset{X}{\times} \wt X \underset{X}{\times} X'
\end{align*}

and 

\begin{align*}
 \wt Y' \underset{Y'}{\times} \wt Y' & = \wt Y'  \underset{X'}{\times} \wt X' \\
&=  (Y\underset{X}{\times} \wt X \underset{X}{\times} X') \underset{X'}{\times} (X' \underset{X}{\times} \wt X) \\
&= Y \underset{X}{\times} \wt X \underset{X}{\times} X' \underset{X}{\times} \wt X
\end{align*}
Furthermore the projections $\wt Y' \lar \wt Y' \underset{Y'}{\times} \wt Y' \rar \wt Y'$ correspond to 
to the maps $(y,\wt x_1,x',\wt x_2)\mapsto (y,\wt x_1,x')$ and $(y,\wt x_1,x',\wt x_2)\mapsto (y,\wt x_2,x')$
under this identification.
Since our description is independent of $Y'$, we may replace $Y'$ by $Y \underset{X}{\times} X'$ in the argument and get 
$$
Hom(Y',T)=\ker(Hom(\wt Y',T) \rightrightarrows Hom(\wt Y' \underset{Y'}{\times} \wt Y', T))=Hom(Y \underset{X}{\times} X',T)
$$
\end{proof}
\end{lem}
Torsors satisfy a variant of the third isomorphism theorem:
\begin{prop}\label{propIterQuot}
 Let $H \triangleleft G$ be a linear algebraic group along with a normal subgroup.
\begin{enumerate}
\item Let $X\rar X/G$ be a $G$-torsor. Then the quotient variety $X/H$ exists. Furthermore $X\rar X/H$ is a $H$-torsor and $X/H\rar X/G$ is a $G/H$-torsor.
\item Suppose in addition that $X\rar X/G$ and $G/H$ are both locally trivial with respect to a topology $\tau$. Then 
$X\rar X/H$ and $X/H\rar X/G$ are both locally trivial with respect to $\tau$ as well.
\end{enumerate}
\begin{proof}
The key idea is to construct $X/H$ by affine descent. The author learned this insight from Laurent Moret-Bailly 
\cite{MO143159}.

\begin{enumerate}
 \item Consider the following diagram:
$$
\begin{xy}
\xymatrix{
(G/H)\times G \times X \ar@/^0.5cm/[rr]^{(q,g,x)\mapsto (q g^{-1}, gx)} \ar@/_0.5cm/[rr]_{(q,g,x)\mapsto (q,x)} \ar[dd]_{(q,g,x)\mapsto (g,x)} &&  (G/H) \times X \ar[dd]^{(q,x)\mapsto x} &  \\
&&& \\
G\times X \ar@/^0.5cm/[rr]^{(g,x)\mapsto gx} \ar@/_0.5cm/[rr]_{(g,x)\mapsto x} && X \ar[r] & X/G
}
\end{xy}
$$
It satisfies the assumptions of \cite[VII.7.9.]{SGA1} and hence we can apply ``affine descent along $X\rar X/G$''.
This produces a scheme $Y$ along with maps $Y\rar X/G$ and $(G/H) \times X \rar Y$ such that 
both of the following squares are cartesian:
$$
\begin{xy}
 \xymatrix{
G\times Q\times X \ar[d]^{(q,x)} \ar[rr]^{(qg^{-1},gx)} && Q \times X \ar[d] \\
Q \times X \ar[rr] \ar[d] && Y \ar[d] \\
X \ar[rr] && X/G
}
\end{xy}
$$
Here and in future we use(d) abbreviations $Q:=G/H$ and $(qg^{-1},gx)$ for the map $(g,q,x) \mapsto (qg^{-1},gx)$ etc.

Our goal is to show that $Y=X/H$ is the desired quotient.

Define $X\rar Y$ to be the map $X\rar Q\times X \rar Y$. Let us show that $X\rar Y$ is a $H$-torsor.
\begin{itemize}
 \item 
 In order to show that $X\rar Y$ is $H$-invariant, we need to check that both paths in the diagram
 $$
 \begin{xy}
 \xymatrix{
  H \times X \ar[r]^{x} \ar[d]_{hx} & X \ar[d]\\
  X \ar[r] & Y
  }
 \end{xy}
 $$
 coincide. To this end we refine it as follows:
 $$
 \begin{xy}
  \xymatrix{ H \times X \ar[rd]^{(h,e,x)} \ar[dd]_{hx} \ar[rr]^x && X \ar[d]^{(e,x)} \\
  & H \times Q \times X \ar[r]^{(q,x)} \ar[d]_{(q,hx)=(qh^{-1},hx)} & Q \times X \ar[d]\\
  X \ar[r]_{(e,x)} & Q \times X \ar[r] & Y
  }
 \end{xy}
 $$
\item In order to see, that 
$$
\begin{xy}
 \xymatrix{H \times X \ar[d]_x \ar[r]^{hx} & X \ar[d] \\
X \ar[r] & Y
}
\end{xy}
$$
is cartesian, one contemplates the diagram:
$$
\begin{xy}
 \xymatrix{H \times X \ar[d]_{(h,x)} \ar[rrr]^{hx} &&& X \ar[d]^{(e,x)} \\
G \times X \ar[d]_x \ar[r]^{(g,e,x)} & G\times Q \times X \ar[d]^{(q,x)}\ar[rr]^{(qg^{-1},gx)} && Q\times X \ar[d] \\
X \ar[r]_{(e,x)} & Q\times X \ar[rr] && Y
}
\end{xy}
$$
\end{itemize}
Next we need to show that $X/H\rar X/G$ is a $(G/H)$-torsor.
So we need to check that the diagram
$$
\begin{xy}
 \xymatrix{
(G/H) \times (X/H) \ar[d]_x \ar[rr]^{gx} &&  X/H  \ar[d] \\
X/H \ar[rr] && X/G
}
\end{xy}
$$
is cartesian. By \ref{lemCartesianFpqcLocal} we may replace it by its pullback along $X\rar X/G$.
The latter is given by 
$$
\begin{xy}
 \xymatrix{
(G/H) \times (G/H) \times X \ar[d]_{(q_2,x)} \ar[rr]^{(q_1 q_2,x)} && (G/H)\times X \ar[d]^x \\
(G/H) \times X \ar[rr]^x && X
}
\end{xy}
$$
and hence cartesian.

Note that until now, we only know that $X/H$ is a scheme. It is even a variety, since $X/H \rar X/G$ is smooth.
Indeed the point is that smoothness can be tested locally in the $fpqc$-topology \cite[2.36]{vistolyfib}.

\item We need to show that $X\rar X/H$ (resp. $X/H\rar X/G$) are locally trivial. Let $(U_i\rar X/G)_{i \in I}$ be a cover over which $X \rar X/G$ trivializes.
Consider the diagram 
$$
\begin{xy}
 \xymatrix{
X \ar[d] & \ar[l] G\times U_i \ar[d] \\
X/H \ar[d] &\ar[l] (G\times U_i)/H \ar[d] \\
X/G &\ar[l] U_i 
}
\end{xy}
$$
The horizontals are maps of torsors, hence both squares are cartesian. We get that $X/H\rar X/G$ is trivial over $U_i$ and furthermore the question
of local triviality of $X\rightarrow X/H$ is reduced to the question of local triviality of $G\times U_i\rar (G/H) \times U_i$. 
Since $G\rar G/H$ is by assumption locally trivial we are done.
\end{enumerate}
\end{proof}
\end{prop}

\subsection{Definition of the equivariant derived category}
Let us give a definition of the equivariant derived category. We start with some motivation. Let $X$ be a variety equipped with an action by 
a linear algebraic group $G$.

Imagine we had some ``classifying space'' of constructible sheaves. By this we mean a space $D_c^b(\E)$ such that 
maps from a variety $X$ to $D_c^b(\E)$ are the same thing as elements of $D_c^b(X,\E)$. 
Imagine also that there existed a quotient $X//G$. Then we could simply define the 
equivariant derived category $D_{G,c}^b(X,\E)$ to be maps form $X//G$ to $D_c^b(\E)$.

Now while $X//G$ and $D_c^b(\E)$ don't make sense in the category of varieties, they exist as fibered categories in a tautological way:
\begin{itemize}
 \item The fiber of $D_c^b(\E)$ over a variety $T$ is $D_c^b(T,\E)$. See \cite[2.4.1]{BernsteinLunts} for a more complete definition.
\item The fiber of $X//G$ over a variety $T$ is the category of $G$-torsors on $T$, equipped with a $G$-equivariant
map to $X$:
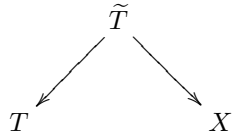
$$
\begin{xy}
 \xymatrix{& \wt T \ar[ld] \ar[rd] &  \\
T & & X
}
\end{xy}
$$
In other words $X//G$ is the quotient stack, see \cite[04WL and 0370]{StacksProj} for a complete definition.
\end{itemize}

\begin{defi}
Let $X$ be a $G$-space, where $G$ is a linear algebraic group.
The equivariant derived category $D^b_{G,c}(X,E)$ is defined as the category of cartesian functors from $X//G$ to $D_c^b(\E)$ such that
the diagram
$$
\begin{xy}
 \xymatrix{X//G \ar[rd] \ar@{-->}[rr]& & D_c^b(\E) \ar[ld] \\
   & Var &
}
\end{xy}
$$
strictly commutes. Here $Var$ is the category of varieties.
\end{defi}
The equivariant derived category is triangulated, idempotent complete category. It is equipped with a canonical ``pullback'' or ``forgetful'' functor:
 $$D^b_{G,c}(X,\E) \rar D^b_c(X,\E)$$
This functor is induced by the canonical map $X\rar X//G$ given by action and projection:
$$
\begin{xy}
 \xymatrix{& G\times X \ar[ld]_{\pi} \ar[rd]^{\rho} &  \\
X & & X
}
\end{xy}
$$
For a $G$-stable stratification $X=\bigsqcup X_\l$ one defines $D^b_{G,\Lambda}(X,\E)$ to be the preimage
of $D^b_\Lambda(X,\E)$ under this functor.

\begin{rem}\label{RemObjectInDbG}
More explicitly an object $\cal F \in D^b_{G,c}(X,\E)$ consists of the following data:

\begin{itemize}
 \item 

For every $f:T\rar X//G$ consisting of a $G$-torsor $\wt T \rar T$ and an equivariant map $\wt T\rar X$:
$$
\begin{xy}
 \xymatrix{& \wt T \ar[ld]_{f_T} \ar[rd]^{f_X} &  \\
T & & X
}
\end{xy}
$$
an object of $D^b_c(T,\E)$, which we name $f^*(\cal F)$.

\item For every $\alpha:f\Rightarrow g$ consisting of a map of $G$-torsor making the diagram 
$$
\begin{xy}
 \xymatrix{
& \wt T \ar[rd] \ar[ld] \ar[dd]^\alpha & \\
T  & & X \\
& \wt T' \ar[ru] \ar[lu] & \\
}
\end{xy}
$$
commutative an isomorphism $\alpha^*:f^* \cal F \rar g^* \cal F$.
\item For every map $g:S\rar T$ and $f:T\rar X//G$ an isomorphism
$$g^*(f^* \cal F)\cong (fg)^* \cal F$$
\end{itemize}
These data satisfy the natural coherence conditions.
\end{rem}

\subsection{Acyclic maps}
Acyclic maps are the relative version of acyclic spaces. If two spaces are connected by an acyclic map, their 
derived categories essentially contain the same information. This will be true, even if one of the spaces is $X//G$,
which explains our interest in the notion.
\begin{defi}
 Let $f:X\rar Y$ be a map of varieties. 
\begin{itemize}
\item We say that $f$ is pre-$n$-acyclic (with respect to $\mathbb E$)
if for every constructible sheaf\footnote{object in the heart of the naive $t$-structure.} of $\E$-modules $\cal F \in Sh_c(Y,\E)$ the map $\cal F \rar \tau^{\leq n} f_* f^* \cal F$ is an isomorphism.

 \item We say that $f$ is $n$-acyclic, if for every cartesian diagram 
$$
\begin{xy}
 \xymatrix{X' \ar[d] \ar[r]^{f'} & Y' \ar[d] \\
X \ar[r]^f & Y
}
\end{xy}
$$
the map $f'$ is pre-$n$-acyclic.
\item We say that $f$ is acyclic, if it is $n$-acyclic for every $n$.
\end{itemize}
\end{defi}
The property of being $n$-acyclic is local in the smooth topology and stable under
the usual operations:
\begin{prop}\label{PropAcyc}
Let 
$$
\begin{xy}
 \xymatrix{X' \ar[d]^\pi \ar[r]^{f'} & Y' \ar[d]^{\pi} \\
X \ar[r]^{f} & Y
}
\end{xy}
$$
be a cartesian diagram. 
\begin{enumerate}
\item If $f$ is $n$-acyclic, then $f'$ is $n$-acyclic as well. 
\item Assume that $\pi$ is smooth and surjective. 
If $f'$ is (pre-)$n$-acyclic, then $f$ is (pre-)$n$-acyclic as well.
\item Let $g:Y\rar Z$ be another map. If $f$ and $g$ are both $n$-acyclic, then $g\circ f$ is $n$-acyclic as well.
\item Let $g:B\rar A$ be another map. If $f$ and $g$ are both $n$-acyclic, then $g\times f$ is $n$-acyclic as well.
\end{enumerate}

\begin{proof}
\begin{enumerate}
 \item Immediate from the definition.
\item Let $\cal F \in Sh_c(X,\E)$. By surjectivity it suffices to check that 
$$\pi^* \cal F \rar \pi^* \tau^{\leq n} f_* f^* \cal F $$
is an isomorphism.
Using smooth basechange we compute:

\begin{align*}
 \pi^* \tau^{\leq n} f_* f^* \cal F &= \tau^{\leq n} \pi^*  f_* f^* \cal F \\
&= \tau^{\leq n} f'_* \pi^*   f^* \cal F \\
&= \tau^{\leq n} f'_*   f'^* \pi^* \cal F \\
&=\pi^* \cal F
\end{align*}

\item Let $\cal G \in Sh_c(X,\E)$. Applying $\tau^{\leq n} g_*$ to the triangle
$$\tau^{\leq n} f_* \cal G \rar f_* \cal G \rar \tau^{>n} f_* \cal G \rar $$ 
gives a triangle 
$$\tau^{\leq n} g_*\tau^{\leq n} f_* \cal G \rar \tau^{\leq n} g_*f_* \cal F \rar \tau^{\leq n} g_*\tau^{>n} f_* \cal G \rar $$ 
Now since $g_*$ is left exact, we have $\tau^{\leq n} g_*\tau^{>n} f_* \cal G=0$. Therefore the first map is an isomorphism 
$$ \tau^{\leq n} g_*  \tau^{\leq n} f_* \cal G=\tau^{\leq n}(g \circ f)_* \cal G$$
Using this isomorphism we can now prove the assertion:
\begin{align*}
 \tau^{\leq n} (g\circ f)_* (g\circ f)^* \cal F &= \tau^{\leq n} g_* \tau^{\leq n} f_* f^* g^* \cal F \\
&= \tau^{\leq n} g_* g^* \cal F \\
&= \cal F
\end{align*}

\item We have $(g\times f)=(g\times id) \circ (id \times f)$.
Since we already know that n-acyclicity is stable under basechange and 
composition, we are done.
\end{enumerate}
\end{proof}
\end{prop}

Let us give a useful criterion for checking that maps are $n$-acyclic:

A fiber bundle $f:X\rar Y$ with fiber $F$ is a map of varieties, such that there exists a $fpcq$-morphism 
$Y'\rar Y$ fitting into a cartesian square 
$$
\begin{xy}
 \xymatrix{F \times Y' \ar[d]_{(f,y')\mapsto y'} \ar[r] & X \ar[d]\\
Y' \ar[r] & Y
}
\end{xy}
$$
We say that a fiber bundle is trivial, if one can choose $Y'\rar Y$ to be the identity. We say that a fiber bundle 
is locally trivial with respect to a topology $\tau$, if there exists a $\tau$-cover of $Y$, such that the bundle 
becomes trivial, when pulled back to the constituents of the cover.

\begin{cor}\label{CorAcycBaseFib}
 Let $f:X\rar Y$ be a fiber bundle with fiber $F$, locally trivial in the smooth topology. 
Suppose that $F$ is pre-$n$-acyclic and $H^{n+1}(F,\E)$ is torsion free. Then $f$ is $n$-acyclic. 
\begin{proof}
 Let $Y\rar X$ be a fiber bundle, with pre-$n$-acyclic fiber $F$. Given a commutative diagram 
$$ 
\begin{xy}
 \xymatrix{
X' \ar[d] \ar[r] & X \ar[d] \\
Y' \ar[r] & Y
}
\end{xy}
$$
the map $X'\rar Y'$ is a fiber bundle with fiber $F$ as well. 
Hence it suffices to show that $X\rar Y$ is pre-$n$-acyclic. By \ref{PropAcyc} this can be checked smooth locally. Hence we may assume 
that the map $f$ is actually the projection $(\pi\times id_X):F\times X \rar X$, where $\pi:F\rar pt$. Thus we may compute:

\begin{align*}
 \tau^{\leq n} f_* f^* \cal F &= \tau^{\leq n} ((\pi\times id_X)_* (\E \boxtimes \cal F)) \\
&=\tau^{\leq n} (H^\bullet(F,\E) \otimes \cal F) \\
&=\cal F
\end{align*}
Note that the last equality uses the torsion freeness assumption.
\end{proof}
\end{cor}

\subsubsection{Application to equivariant derived categories}
Acyclic maps help us to compute the equivariant derived category. In order to be more precise, we 
introduce some notation.

Let $X$ be a $G$-space and $f:T\rar X//G$ be a map given by
$f_T:\wt T \rar T$ a $G$-torsor and $f_X:\wt T \rar X$ equivariant.
Denote by $D^b_c(X,\wt T,\E)$ the category whose objects consist of triples 
$$(\cal F_X, \eta_{\cal F}, \cal F_T)$$
where $\cal F_X$ is an object of $D^b_c(X,\E)$, $\cal F_T$ is an object of $D^b_c(T,\E)$ and $\eta_{\cal F}: f_T^* \cal F_T \cong f_X^* \cal F_X$ is an 
isomorphism.
Morphisms in $D^b_c(X,\wt T,\E)$ are pairs of maps $\psi_T:\cal F_T \rar \cal G_T$ and $\psi_X:\cal F_X \rar \cal G_X$
making the obvious diagram commute.

Let $g:S\rar X$ be a second map and $\alpha:f \Rightarrow g$ be a ``transformation'' consisting of a map of torsors 
$\wt \alpha: \wt T \rar \wt S$ making the diagram 
$$
\begin{xy}
 \xymatrix{T \ar[d]^{\alpha} & \ar[l] \wt T  \ar[d]^{\wt \alpha} \ar[rd] & \\
 S & \ar[l] \wt S \ar[r] & X
 }
\end{xy}
$$
commutative. Then there is a canonical pullback functor:
$$\alpha^*: D^b_c(X,\wt S,\E) \rar D^b_c(X,\wt T,\E)$$

%

Let $I=[a,b]:=[a,a+1,\ldots,b]\subset \Z$. Then we denote by $D_c^I(X,\E)$ the full subcategory 
of $D^b_c(X,\E)$ consisting of those objects whose cohomology sheaves vanish away from $I$:
$$D_c^I(X,\E):=\{\cal F \in D_c^b(X,\E)|(H^i(\cal F) \neq 0)\Rightarrow i \in I \}$$
We also use variants of this notations like $D^I_{G,c}(X,\E)$ and $D^I_c(X,\wt T,\E)$ 
referring to objects whose cohomology sheaves vanish outside of $I$ on any involved space.

\begin{prop}\label{propComputationOfEquivDer}
Let $I=[a,b]$ with $0 \leq n:=b-a$. Suppose that we are given $f:T\rar X//G$ such that $f_X$ is $n$-acyclic.
\begin{enumerate}
 \item Then there is a canonical pullback functor $D^b_{G,c}(X,\E) \rar D^b_c(X,\wt T,\E)$, which restricts to an equivalence 
$$D^I_{G,c}(X,\E) \rar D^I_c(X,\wt T,\E)$$
\item Let $g:S\rar X//G$ be another map and $\alpha:f\Rightarrow g$ be a transformation.
Then the diagram
$$
\begin{xy}
 \xymatrix{ & \ar[ld] D^b_{G,c}(X,\E) \ar[rd] & \\
 D^b_c(X,\wt T,\E) && \ar[ll]^{\alpha^*} D^b_c(X,\wt S,\E)
 }
\end{xy}
$$
commutes.
\end{enumerate}
\begin{proof}
\begin{enumerate}
 \item This is a reformulation of \cite[2.4.3.]{BernsteinLunts}. It relies on the torsor-formalism, 
especially \ref{propIterQuot}.
\item Straightforward.
\end{enumerate}
\end{proof}
\end{prop}

\section{Approximations}
In the light of \ref{RemObjectInDbG} the category $D^b_{G,c}(X,\E)$ is not very transparent. In order to 
describe it more explicitly, we will need approximations.

Given a stratified space $(X,\Lambda)$ and a say smooth $n$-acyclic surjection $\pi:Y\rar X$,
$D^I_\Lambda(X,\E)$ and $D^I_\Lambda(Y,\E)$ are equivalent via $f^*$ and $\tau^{\leq n} f_*$. 
Here $I=[a,b]$ and $n=b-a$. Hence if we find a sequence of maps 
$$
\begin{xy}
\xymatrix{
 & & X  \\
X_0 \ar[r] \ar[urr] & X_1 \ar[r] \ar[ur] & X_2 \ar[r]  \ar[u] &\ldots \\
}
\end{xy}
$$
such that $X_n\rar X$ is smooth and $n$-acyclic, we should be able to compute $D^b_\Lambda(X)$ by the formula:
$$D^b_\Lambda(X,\E)= \varprojlim D^b_\Lambda(X_n,\E)$$
Here the transition functors for the limit are given by pullback along the horizontals $X_i\rar X_{i+1}$.
Now it is of course much easier to compute $D^b_\Lambda(X,\E)$ directly. However this method has the advantage, 
that it also works if we replace $X$ by a more general space, for example $X//G$. Recall that a map $T\rar X//G$ 
is given by a roof $T \lar \wt T \rar X$. This suggests that a suitable diagram
$$
\begin{xy}
\xymatrix{
 & & X  \\
X_0 \ar[r] \ar[urr] \ar[d] & X_1 \ar[r] \ar[d] \ar[ur] & X_2 \ar[r] \ar[d] \ar[u] &\ldots \\
\ol X_0 \ar[r] &  \ol X_1 \ar[r] & \ol X_2 \ar[r] & \ldots \\
}
\end{xy}
$$
allows one to compute the equivariant derived category via:
$$D^b_{G,\Lambda}(X,\E)= \varprojlim D^b_\Lambda(\ol X_n,\E)$$

This formula holds under quite weak assumptions on the diagram above,
see \cite[Section 5.3.]{OlafThesis}.
However we will put additional conditions into our notion of approximation, which make it easier to control $\varprojlim D^b_\Lambda(X_n,\E)$.

\subsection{Principal bundles and balanced products}
Let $G$ be a linear algebraic group. Let us fix some terminology:
 \begin{itemize}
\item Let $\pi:Y \rar B$ be a map of varieties. We say that $\pi$ is a locally trivial bundle with fiber $F$,
if there exists a covering by Zariski open
\footnote{This is a quite strong condition, for example $G\rightarrow G/H$ is not necessarily Zariski locally trivial.}
 subsets $B=\bigcup U_i$ such that there 
are commutative diagrams 
$$
\begin{xy}
 \xymatrix{\pi^{-1}(U_i) \ar[rd]^\pi \ar[r]^{\cong} & F \times U_i \ar[d]^{\pi_{U_i}} \\
 & U_i}
\end{xy}
$$
whose horizontal is an isomorphism.
\item Similarly by a principal $G$-bundle, we mean a principal $G$-bundle which is locally trivial in the Zariski topology. 
In other words a principal $G$-bundle is a Zariski locally trivial $G$-torsor.
\item Let $X$ and $Y$ two $G$-varieties. 
We define the balanced product by the formula 
$$X \underset{G}{\otimes}Y:=(X\times Y)/G$$
if this quotient exists. Here $G$ acts diagonally.
 \end{itemize}
First we need a criterion for balanced products to exist:
\begin{prop}\label{PropBalancedProdExists}
Let $E\rar B$ be a principal $G$-bundle. Let $X$ be a variety with a $G$-action. 
Then the balanced product
$$E\underset{G}{\otimes} X$$
 exists as a variety. The map $B \lar E\underset{G}{\otimes} X$ is a locally trivial bundle with fiber $X$.
Furthermore
\begin{itemize}
\item If $B$ and $X$ are both acyclic, then $E\underset{G}{\otimes} X$ is acyclic as well.
\item If $B$ and $X$ are both affine, then $E\underset{G}{\otimes} X$ is affine as well.
\item If $B$ and $X$ are both smooth, then is $E\underset{G}{\otimes} X$ smooth as well.
\end{itemize}

\begin{proof}
 If $E$ is trivial, we have $E\underset{G}{\otimes} X=B\times X$, which is certainly a variety. Hence the quotient locally exists. 
We can glue it together and obtain a prevariety $E\underset{G}{\otimes} X$. It remains to check that $E\underset{G}{\otimes} X$ is separated.
Since $B$ is separated, we only need to check that the map $B \lar E\underset{G}{\otimes} X$ is separated. But this can be checked locally on $B$, see 
\cite[II:4.6]{Hartshorne}.

Checking that various properties are preserved, is analogous using
appropriate relative notions \cite[III.10.1]{Hartshorne} , \cite[II.5.17]{Hartshorne}
and \ref{PropAcyc}.
\end{proof}
\end{prop}

\begin{lem}\label{lemStratOnBalancedProd}
 Let $\pi:E\rar B$ be a principal $G$-bundle over a stratified variety $(B,\Sigma)$.
Let $(X,\Lambda)$ be a stratified variety with compatible $G$-action.
\begin{enumerate}
 \item Then $E\underset{G}{\otimes} X$ is again stratified with strata $E_\sigma \underset{G}{\otimes} X_\l:=\pi^{-1}(B_\sigma) \underset{G}{\otimes} X_\l $:
$$E\underset{G}{\otimes} X=\bigsqcup_{(\sigma,\l) \in \Sigma \times \Lambda} E_\sigma \underset{G}{\otimes} X_\l$$
\item If the stratifications on $B$ and $X$ are both acyclic, then the stratification on $E\underset{G}{\otimes} X$
is acyclic as well.
\end{enumerate}

\begin{proof}
\begin{enumerate}
\item First of all our candidate strata $E_\sigma \underset{G}{\otimes} X_\l$ are smooth by \ref{PropBalancedProdExists}.

The stratification on $B$ induces a stratification on $E$ 
and hence we get a stratification on $E\times X$ with strata $E_\sigma \times X_\l$.

We need to show, that say $l^*_{\sigma',\l'} l_{\sigma,\l *}\E$ is constant for 
inclusions 
$$E_{\sigma'} \underset{G}{\otimes} X_{\l'} \inj E\underset{G}{\otimes} X \linj E_\sigma \underset{G}{\otimes} X_\l$$
Now the property of being a local system can be tested smooth locally (see  \ref{lemParityCritPullback}).
Hence exploiting the diagram of torsors
$$
\begin{xy}
 \xymatrix{ E_{\sigma'} \times X_{\l'} \ar[d] \ar[r] & E\times X  \ar[d] & \ar[l] E_\sigma \times X_\l \ar[d] \\
E_{\sigma'} \underset{G}{\otimes} X_{\l'}  \ar[r] & E\underset{G}{\otimes} X  & \ar[l] E_\sigma \underset{G}{\otimes} X_\l
}
\end{xy}
$$
we see that $l^*_{\sigma',\l'} l_{\sigma,\l *}\E$ is a local system. 

We need to find an open dense subset of $E_{\sigma'} \times X_{\l'}$ on which $l^*_{\sigma',\l'} l_{\sigma,\l *}\E$ is trivial.
We may assume that $E_{\sigma'} \times X_{\l'}$ lies in the closure of $E_\sigma \times X_\l$, otherwise $l^*_{\sigma',\l'} l_{\sigma, \l *}\E=0$ anyway.
Since our bundles are Zariski locally trivial, there exists an open dense subset $U \subset B$ over which $E$ is trivial and such that $U$ intersects 
$B_{\sigma'}$ and hence also $B_{\sigma}$. 
It is not hard to see that $l^*_{\sigma',\l'} l_{\sigma,\l *}\E$ is constant on the preimage of $U$.

\item \ref{PropBalancedProdExists}
\end{enumerate}
\end{proof}
\end{lem}

\begin{rem}
\begin{itemize}
 \item 

Let $B$ and $X$ be stratified varieties and assume that $X$ comes with a compatible $G$-action.
Let $E\rar B$ be the trivial $G$-bundle
Then we have $E\underset{G}{\otimes} X=B\times X$ as stratified varieties. 

In particular for any principal bundle $E\rar B$ the balanced product $E \underset{G}{\otimes} X$ looks locally over 
$B$ like a product stratification.

\item There are various ways to realize partial flag varieties as bundles over smaller flag varieties, with fibers
other flag varieties. However in this situation the stratification does NOT locally look like a product stratification.
This makes the category of perverse sheaves on the full flag variety very complicated.
\end{itemize}
\end{rem}

\subsection{Normally smooth inclusions}
Let $i:Z\inj X$ be a closed codimension $c$ embedding of varieties. We want to find a condition on $i$, such that the
pullback preserves the $IC$-extension of the constant sheaf on the smooth locus up to appropriate shift:
$$i^* IC_X \cong IC_Z[c]$$
This is for example true if we have an inclusion of a slice $Z\rar Z\times S$ for smooth $S$. More generally, 
it should hold, if $X$ is smooth in the direction perpendicular to $Z$. 

\begin{defi}\label{defiNormallySmooth}
We define the class of normally smooth inclusions (of codimension $\bullet$) to be the smallest system of $\N$-indexed classes of morphisms, 
satisfying the following axioms:
\begin{enumerate}
\item Let $i:Z\rar X$ be a closed immersion between smooth varieties, such that $\dim X - \dim Z =c$. Then $i$ is normally smooth inclusions of codimension $c$.
\item Let $i:Z\rar X$ be an isomorphism. Then $i$ is a normally smooth of codimension zero.
\item Let $i:Z\rar X$ and $i:Z'\rar X'$ be normally smooth of codimensions $c$ and $c'$. Then $i\times i'$ is normally smooth of 
codimension $c+c'$.
\item Let $i:Y\rar Z\rar X$ be a concatenation of two normally smooth inclusions of codimensions $c$ and $c'$. Then $i$ is normally smooth of
codimension $c+c'$.
\item Let $(U_i\rar X)_{i\in I}$ be an open cover of $X$ and $\phi:Z\rar X$ be a map. If all $\phi^{-1}(U_i) \rar U_i$
are normally smooth of codimension $c$, then the same holds for $\phi$.
\item\label{ItemSmoothPull} Let $X'\rar X$ be smooth, $i:Z\rar X$ be normally smooth of codimensions $c$ and the following diagram be cartesian:
$$
\begin{xy}
   \xymatrix{
Z' \ar[d] \ar[r]^{i'} & X' \ar[d]\\
Z \ar[r]^i & X
}
\end{xy}
$$
Then $i'$ is normally smooth of codimension $c$ as well.
\end{enumerate}
\end{defi}
Note that all normally smooth inclusions are automatically closed immersions. The point is that 
the class of closed immersions is stable under all operations listed in the definition.
Let us give the most important example for our purposes:
\begin{lem}\label{lemBalProdNormSm}
 Let 
$$
\begin{xy}
 \xymatrix{E \ar[r] \ar[d] & E' \ar[d] \\
B \ar[r] & B'
}
\end{xy}
$$
be a cartesian diagram, where $E'\rar B'$ is a principal $G$-bundle and let $X$ be a $G$-variety. 
If $B\inj B'$ is normally smooth of codimension $c$, then so is 
$$E\underset{G}{\otimes} X \inj E'\underset{G}{\otimes} X$$
\begin{proof}
Let $(U_i'\rar B')_{i \in I}$ be an open cover over which $E'$ is trivial and $(U_i\rar B)_{i \in I}$ its pullback. Then 
$E\underset{G}{\otimes} X \rar E'\underset{G}{\otimes} X$ is locally isomorphic to 
$U_i\times X \rar U_i' \times X$.
\end{proof}
\end{lem}

\begin{lem}\label{lemBoXProd}
 Let $X,Y$ two stratified varieties
and equip $X\times Y$ with the product stratification. 
\begin{enumerate}
 \item Then $\boxtimes:D^b_{\Lambda}(X,\E) \times D^b_{\Lambda'}(Y,\E) \rar D^b_{\Lambda\times \Lambda'}(X\times Y,\E)$ is 
right exact amplitude one with respect to both the naive and the perverse t-structure. If $\E$ field, $\boxtimes$ is even exact.
\item We have $IC_\l \boxtimes IC_{\l'}=IC_{(\l,\l')}$
\end{enumerate}

\begin{proof}
Since the six functors commute with $\boxtimes$, an investigation of stalks yields the claims about amplitude and (right)exactness 
with respect to the naive t-structure.
 
We proof the other assertions inductively
and consider only the most difficult case $\E=\O$. If $X,Y$ both consist of a single stratum, 
the perverse and the naive t-structure coincide up to shift. Hence all claims hold in this case.

Assume we have proven the assertions for varieties, where $X$ has less then $n$-strata and $Y$ has less then $m$ strata.
By symmetry we only have to show that it also holds when we add a single open stratum to $X$.
Let $j:U\inj X$ be the new stratum and $i:Z\inj X$ be its complement.
\begin{enumerate}
 \item 

Recall that we have $\cal F \in \prescript{\mathfrak{p}}{}{D}^{\leq 0}(X,\O)$ if and only if $j^* \cal F \in \prescript{\mathfrak{p}}{}{D}^{\leq 0}(U,\O)$ and $i^* \cal F \in  \prescript{\mathfrak{p}}{}{D}^{\leq 0}(Z,\O)$.
Now let $\cal F \in \prescript{\mathfrak{p}}{}{D}^{\leq 0}(X,\O)$ and $\cal G \in \prescript{\mathfrak{p}}{}{D}^{\leq 0}(Y,\O)$.
We compute and use the induction hypothesis:
$$(j \times id_Y)^* (\cal F \boxtimes \cal G)= (j^* \cal F) \boxtimes \cal G \in \prescript{\mathfrak{p}}{}{D}^{\leq 0}(U\times Y,\O)$$
and 
$$(i \times id_Y)^* (\cal F \boxtimes \cal G)= (i^* \cal F) \boxtimes \cal G \in \prescript{\mathfrak{p}}{}{D}^{\leq 0}(Z\times Y,\O)$$
But these two lines exactly contain the conditions for $\cal F \boxtimes \cal G$ to be in $\prescript{\mathfrak{p}}{}{D}^{\leq 0}(X\times Y,\O)$.
The assertion about amplitude one is proven in a similar way.

\item We want to show $IC_\l \boxtimes IC_{\l'}=IC_{(\l,\l')}$. Without loss of generality we may assume that 
$X_\l=U$ and $Y_\l'=V$ are open. Consider the open closed decompositions:
$$U \inj X \linj F \text{ and } V \inj Y \linj G$$

By \cite[1.4.24.]{BBD} the intermediate extension $IC_\l$ is uniquely characterized by the property $j^* IC_\l =\O[d_\l]$ and $i^* IC_\l \in  \prescript{\mathfrak{p}}{}{D}^{\leq -1}(Z,\O)$
while $i^! IC_\l \in  \prescript{\mathfrak{p}}{}{D}^{\geq 1}(Z,\O)$.
We need to check that $IC_\l \boxtimes IC_{\l'}$ inherits these conditions. 
This can be done by using that $\boxtimes$ has amplitude one and decomposing $X\times Y$ successively into
$$X\times Y=(U\times V) \sqcup (U\times G) \sqcup (F\times V) \sqcup (F\times G)$$

\end{enumerate}

\end{proof}
\end{lem}

Given a variety $X$, we denote by $IC_X=IC_X^\E$ the $IC$-sheaf corresponding to the constant local system $\E$ on the smooth locus.
\begin{prop}\label{PropICclosed}
 Let $i:Z\rar X$ be normally smooth of codimension $c$. Then we have
$$i^* IC_X \cong IC_Z[c]$$
\begin{proof}
 We need to show, that the property $i^* IC_X \cong IC_Z[c]$ of a morphism is closed under the operations in \ref{defiNormallySmooth}:
\begin{enumerate}
 \item \cite[16.7]{MilneEtale}
\item Trivial
\item This follows from $IC_{X\times X'}=IC_X \times IC_{X'}$ by \ref{lemBoXProd}.
\item Trivial
\item Straightforward.
\item This follows, since $IC$-complexes are preserved by smooth pullbacks (up to appropriate shift).
\end{enumerate}
\end{proof}
\end{prop}

\subsection{Approximations}
Let $X$ be a variety equipped with a $G$-action. We say that a stratification on $X$ is $G$-stable or that the action is compatible, 
if each stratum is preserved under the $G$-action.
We now have all notions needed to formulate the definition of an approximation.
\begin{defi}
 Let $X=\bigsqcup_{\l \in \Lambda} X_\l$ be a $G$-variety equipped with a $G$-stable acyclic stratification. An approximation $(X_n,X,G)$ of $X$ consists of 
a commutative diagram of varieties
$$
\begin{xy}
\xymatrix{
 & & X  \\
X_1 \ar[r] \ar[urr] \ar[d] & X_2 \ar[r] \ar[d] \ar[ur] & X_3 \ar[r] \ar[d] \ar[u] &\ldots \\
\ol X_2 \ar[r] & \ol X_2 \ar[r] & \ol X_3 \ar[r] & \ldots \\
}
\end{xy}
$$
such that
\begin{itemize}
\item Each $X_n$ is equipped with a $G$-action and the maps $X_n\rar X$ are smooth, $n$-acyclic and $G$-equivariant.
\item The maps $X_n \rar X_{n+1}$ are closed inclusions and equivariant.
\item The maps $X_n\rar \ol X_n$ are principal $G$-bundles.

\item Each $\ol X_n$ is equipped with a stratification indexed by $\Lambda$ and a refinement into an acyclic stratification 
indexed by some $\wt \Lambda_n$.
The maps $\ol X_n \rar \ol X_{n+1}$ are maps of stratified varieties with respect to the stratifications $\wt \Lambda_n,\wt \Lambda_{n+1}$.

\item The stratification on $X_n$ induced by $(X,\Lambda)$ and by $(\ol X_n,\Lambda)$ coincide.
\item For all $\l\in \Lambda$ the maps $\ol X_n\rar \ol X_{n+1}$ induce normally smooth inclusions, between 
strata closures:
$$\ol {\ol X}_{n,\l} \inj \ol {\ol X}_{n+1,\l}$$
\end{itemize}
\end{defi}

\begin{rem}
There are various variants, which inspired these conditions, see \cite{OlafThesis}.
\end{rem}

Let us give some examples:
\begin{ex}\label{exApproxGln}
 \begin{itemize}
\item Let $X=pt$ and $G=\mathbb G_m$. Then 
$$
\begin{xy}
\xymatrix{
 & & pt  \\
(\mathbb A^1-0) \ar[r] \ar[urr] \ar[d] & (\mathbb A^2-0) \ar[r] \ar[d] \ar[ur] & (\mathbb A^3-0) \ar[r] \ar[d] \ar[u] &\ldots \\
\mathbb P^0 \ar[r] & \mathbb P^1 \ar[r] & \mathbb P^2 \ar[r] & \ldots \\
}
\end{xy}
$$
is an approximation.
  \item More generally let $X=pt$ and $G=GL_k$. Denote by $Gr(k,n)$ the Grassmannian and by $E(k,n)$ be the set of $k$-tuples
of linearly independent vectors (i.e. the $E(k,n)$ are Stiefel varieties). Then 
$$
\begin{xy}
\xymatrix{
 & & pt  \\
E(k,k) \ar[r] \ar[urr] \ar[d] & E(k,k+1) \ar[r] \ar[d] \ar[ur] & E(k,k+2) \ar[r] \ar[d] \ar[u] &\ldots \\
Gr(k,k) \ar[r] & Gr(k,k+1) \ar[r] & Gr(k,k+2) \ar[r] & \ldots \\
}
\end{xy}
$$
is an approximation. Over the complex numbers, it is well known, that the Stiefel manifold $E(k,n)$ is $2k$-connected, 
hence $2k$-acyclic. Using \cite[6.1.9.]{BBD}, we see that $E(k,n)$ is also $2k$-acyclic in our setting.
\end{itemize}
\end{ex}
It will be convenient to have ways of constructing new approximations out of old ones:
\begin{thm}\label{thmApproxConstr}
\begin{enumerate}
\item Let $(X_n,X,G)$ and $(Y_n,Y,H)$ be two approximations. Then $(X_n\times Y_n,X\times Y,G\times H)$ is an approximation.
\item Let $N \inj P \surj L$ be a split short exact sequence of algebraic groups. Fix a splitting $L\inj P$.
and suppose that $N$ is acyclic. 
If $(E_n,pt,L)$ is an approximation, then $(P\underset{L}{\otimes} E_n,pt,P)$ is an approximation as well.
\item Let $X$ be a variety with a $G$-action and compatible acyclic stratification. Let $(E_n,pt,G)$ be an approximation. 
Then $(E_n \times X,X,G)$ is an approximation as well, where $G$ acts diagonally on $E_n \times X$.
\end{enumerate}
\begin{proof}
 \begin{enumerate}
\item All involved properties are stable under products.
\item We need to check that $P\underset{L}{\otimes} E_n$ is still $n$-acyclic. Using the multiplication map 
and the splitting we get an $L$-equivariant isomorphism of varieties $N\times L\rar P$. Hence we may compute:
$$ 
P\underset{L}{\otimes} E_n \cong (N \times L)\underset{L}{\otimes} E_n =N \times E_n
$$
But the product of two $n$-acyclic maps is still $n$-acyclic.
\item 
\ref{PropAcyc} gives us acyclicity
and \ref{lemBalProdNormSm} guarantees the normally smoothness conditions.
\end{enumerate}
\end{proof}
\end{thm}

\begin{ex}
Let $X$ be a variety with $G$-action and compatible acyclic stratification. Combining \ref{exApproxGln} and \ref{thmApproxConstr} 
we see that there exist approximations $(X_n,X,G)$ in the following cases:
\begin{itemize}
\item $G=T$ is a torus or more generally a product of some $GL_{n_i}$.
\item $G$ is a parabolic subgroup of $GL_n$.
\item $G$ is a connected solvable linear algebraic group. Indeed the point is that $G$ can be written as a semidirect product of 
a maximal torus $T$ and its unipotent elements $G_u$. Furthermore $G_u$ is an iterated extension of additive groups 
and hence acyclic. See \cite[10.6]{Borelalggrp} for these facts.
\end{itemize}
\end{ex}
\subsection{The equivariant derived category as a limit}
Recall the notion of an inverse limit of categories from  \cite{OlafThesis}.
Let 
$$
\begin{xy}
 \xymatrix{\cal C_0 & \ar[l]_{F_1} \cal C_1 & \ar[l]_{F_2} \cal C_2 & \ar[l]_{F_3} \ldots
}
\end{xy}
 $$
be a sequence of categories and functors. We define the category $\varprojlim \cal C_n$ in the following way:
\begin{itemize}
 \item Objects consist of families $(X_i,\psi_i)$, where 
$X_i \in \cal C_i$ and $\psi_i:F_i(X_{i+1}) \rar X_i$ is an isomorphism.
\item Morphisms $f: (X_i,\psi_i) \rar (X'_i,\psi'_i)$ consist of maps $f_i:X_i\rar X_i'$ such that
$$
\begin{xy}
 \xymatrix{F_i(X_{i+1}) \ar[d]_{F_i(f_{i+1})} \ar[r]^{\psi_i} & X_i \ar[d]^{f_i} \\
F_i(X'_{i+1})  \ar[r]^{\psi'_i} & X'_i  
}
\end{xy}
$$
commutes.
\end{itemize}
\begin{thm}\label{thmOlafApprox}
 Let $(X,\Lambda)$ be an acyclically stratified variety with compatible $G$-action. Let $(X_n,X,G)$ be an approximation. Then we have 
\footnote{See \ref{subsecInverseLimits} for the definition of $\varprojlim$.}
$$\varprojlim D^b_{\Lambda}(\ol X_n,\E) \cong D^b_{G,\Lambda}(X,\E)$$
\end{thm}
\begin{proof}[Proof of \ref{thmOlafApprox}]
Consider the diagram 
$$
\begin{xy}
 \xymatrix{
X_0 \ar[rr]^i \ar@{=}[d] && X_1 \ar[rr]^i \ar@{=}[d] && X_2 \ar@{=}[d] & \ldots \\
X_0   && X_1   && X_2   & \ldots \\
& X_0\underset{X}{\times} X_1 \ar[lu] \ar[ru]  && X_1 \underset{X}{\times} X_2 \ar[lu] \ar[ru]  && \ldots   \\
}
\end{xy}
$$ 
It does not commute. 
However by \ref{LemTransitionMaps} and \ref{propComputationOfEquivDer} it will after applying $D^I_\Lambda(X,\_,\E)$ to it and dropping the first $|I|+1$ terms. 
Hence we may compute
$$\varprojlim D^I_{\Lambda}(\ol X_n,\E)$$
using $(\pi_{X_k}^*)^{-1}\pi_{X_{k+1}}^*$ as transition functors instead of $i^*$.
Furthermore this replacement is compatible with enlarging $I$. Let us distinguish the two ways of taking the limit notationally 
by 
$$\varprojlim D^I_{\Lambda}(X,X_n,\E) | i^* \text{ and  } \varprojlim D^I_{\Lambda}(X,X_n,\E) | (\pi_{X_k}^*)^{-1}\pi_{X_{k+1}}^*$$
Now we compute:
\begin{align*}
\varprojlim D^b_{\Lambda}(\ol X_n,\E) =& \bigcup_I \varprojlim D^I_{\Lambda}(\ol X_n,\E) & \\
=& \bigcup_I \varprojlim D^I_{\Lambda}(X,X_n,\E) | i^* &  \cite[Prop. 95]{OlafThesis} \\
=& \bigcup_I \varprojlim D^I_{\Lambda}(X,X_n,\E) | (\pi_{X_k}^*)^{-1}\pi_{X_{k+1}}^* & \ref{LemTransitionMaps}\\
=& \bigcup_I D^I_{G,\Lambda}(X,\E) & \ref{propComputationOfEquivDer}\\
=& D^b_{G,\Lambda}(X,\E) &\\
\end{align*}
\end{proof}

\begin{lem}\label{LemTransitionMaps}
Let $f:T\rar X//G$ and $g:T'\rar X//G$ be two maps, for which $f_X$ and $g_X$ are both smooth and $n$-acyclic.
Let $i:f \Rightarrow g$ be a transformation.
$$
\begin{xy}
 \xymatrix{ & X \\
\wt T \ar[d] \ar[r]^i \ar[ur] & \wt T' \ar[u] \ar[d] \\
T \ar[r] & T'
}
\end{xy}
$$
Let $I=[a,b]$ with $b-a \leq n$. Then there exists an isotransformation making the
following diagram of equivalences of categories commutative:
$$
\begin{xy}
 \xymatrix{
& D^I_c(X, \wt T \underset{X}{\times} \wt T',\E)   & \\
D^I_c(X,\wt T,\E) \ar[ru]^{\pi_1^*} && \ar[ll]^{i^*} D^I_c(X,\wt T',\E) \ar[lu]_{\pi_2^*}
}
\end{xy}
$$
\begin{proof}
 Consider the diagram 
$$
\begin{xy}
\xymatrix{
 & \wt T \underset{X}{\times} \wt T' \ar[ld]_{\pi_1} \ar[rd]^{\pi_2} & \\
\wt T \ar[rr]^i && \wt T' 
}
\end{xy}
$$
The problem is, that it does not commute. However we can enlarge it:
$$
\begin{xy}
\xymatrix{
 & \wt T \underset{X}{\times} \wt T' \ar[ld]_{\pi_1} \ar[rd]^{\pi_2} & & \ar[ll]_{id \underset{X}{\times} i} \wt T \underset{X}{\times} \wt T\\
\wt T \ar[rr]^i && \wt T' &
}
\end{xy}
$$

Now we have indeed 
$$i \circ \pi_1 \circ (id \times i) = \pi_2 \circ (id \times i)$$
Forming (truncated) derived categories, we get a diagram

\begin{xy}
\xymatrix{
 & D^I_c(X,\wt T \underset{X}{\times} \wt T',\E)   \ar[rr]^{(id \underset{X}{\times} i)^*} && D^I_c(X,\wt T \underset{X}{\times} \wt T,\E) \\
D^I_c(X,\wt T,\E) \ar[ru]^{\pi_1^*} && \ar[ll]^{i^*}  D^I_c(X, \wt T',\E)  \ar[lu]_{\pi_2^*}&
}
\end{xy}
along with a natural isomorphism
$$(id \times i)^* \circ \pi_1^* \circ i^*= (id \times i)^* \circ \pi_2^* $$
By \ref{propComputationOfEquivDer} all functors in the diagram are equivalences and hence we get:
$$\pi_1^* \circ i^* \cong \pi_2^*$$
\end{proof}
\end{lem}

\section{Perverse sheaves on balanced products}
Let $E\rar B$ be a principal $G$-bundle and $X$ be a $G$-space. Assume that 
$E$ and $X$ are stratified in a way that is compatible with the $G$-action.

The main theme of this section will be that perverse sheaves on the balanced product $(E\times X)/G$ behave like perverse sheaves on $(E/G) \times X$.
More precisely we want to proof the following theorem:
\begin{thm}\label{thmBPPerv}
Let $(B,\Sigma)$ and $(X,\Lambda)$ be acyclically stratified varieties and $X$ be equipped with a compatible $G$-action.
Let $E\rar B$ be a principal $G$-bundle.
Then the following hold:
\begin{enumerate}
\item The space $E\underset{G}{\otimes} X=\bigsqcup_{\sigma,\l \in \Sigma \times \Lambda} E_\sigma \underset{G}{\otimes} X_{\l}$ is an acyclically stratified variety
\item If $B$ and $X$ are $IC^\O$-parity, then so is $E\underset{G}{\otimes} X$.
\item If $B$ and $X$ satisfy the BGS-condition, then so does $E\underset{G}{\otimes} X$.
\item The multiplicities $[\Delta_{\sigma,\l}:IC_{\sigma',\l'}]$ coincide 
in $\P_{\Sigma \times \Lambda}(E\underset{G}{\otimes} X,\K)$ and $\P_{\Sigma \times \Lambda}(B\times X,\K)$.
If $B\times X$ satisfies the BGS-condition, then they also coincide as graded multiplicities.
\end{enumerate}
\begin{proof}
The idea of proof will be to consider first the case where the bundle is trivial $E=G\times B$ and then 
use that the assertions of the theorem are local. We split it into a couple of lemmata:
 \begin{enumerate}
\item \ref{lemStratOnBalancedProd}
\item \ref{lemICBalProdParity}
\item \ref{lemBalProdBGSCond} 
\item \ref{lemBalancedProdMult}
 \end{enumerate}
\end{proof}
\end{thm}
 The essence of \ref{thmBPPerv} is contained in the following statement:
\begin{cor}\label{corBPPerv}
 Let $(B,\Sigma)$ and $(X,\Lambda)$ be acyclically stratified varieties and $X$ be equipped with a compatible $G$-action.
Let $E\rar B$ be a principal $G$-bundle.
If $(B,\Sigma)$ and $(X,\Lambda)$ both satisfy the BGS-condition and $IC^\O$-parity, then $E\underset{G}{\otimes} X$ does also. Furthermore
$$wt(E\underset{G}{\otimes} X)=wt(B)\cdot wt(X)$$
\end{cor}

\subsubsection{Locality of $IC^\O$-parity}

\begin{lem}\label{lemParityCritPullback}
 Let $\pi:X\rightarrow Y$ be a smooth surjection and $\cal F\in Sh_c(Y,\E)$ be a constructible sheaf on $Y$.
\begin{enumerate}
\item Our sheaf $\cal F$ is a local system, if and only if $\pi^*(\cal F)$ is.
\item Our sheaf $\cal F$ is $?$-even if and only if $\pi^* \cal F$ is $?$-even.
\end{enumerate}
Here $? \in \{!,*\}$.
\begin{proof}
\begin{enumerate}
 \item Assume that $\pi^*(\cal F)$ is a local system. If $\pi$ is even étale, the assertion holds.
Now all smooth surjections admit sections étale locally \cite[3.26]{MilneBook}.
More precisely for any smooth surjection $\pi:X\surj Y$, there exists an étale surjection $Y'\surj Y$ and 
$i:Y'\rar X$ making the diagram
$$
\begin{xy}
 \xymatrix{
 & X \ar[d]^\pi \\
Y' \ar[r] \ar@{-->}[ru]^i & Y
}
\end{xy}
$$
commutative.
Hence étale locally
we see that 
$$\cal F|_{Y'} \cong i^* (\pi^* \cal F)$$ 
is the pullback of a local system. Hence $\cal F$ is a local system. 
\item Trivial.
\end{enumerate}
\end{proof}
\end{lem}

\begin{lem}\label{lemICBalProdParity}
Let $E\rar B$ be a principal $G$-bundle, $B,X$ be stratified varieties and $X$ be such that the $G$-action preserves the stratification on $X$.
If $E$ and $X$ satisfy $IC^\O$-parity, then so does $E\underset{G}{\otimes} X$.
\begin{proof}
The $IC$-sheaves on a product are the exterior products of $IC$-sheaves on the factors by \ref{lemBoXProd}.
Hence the $IC$-sheaves on $E\times X$ are parity. 
The pullback map 
$$E\times X \rar E\underset{G}{\otimes} X $$
preserves $IC$-sheaves by smoothness 
and being parity can be tested after a smooth pullback by \ref{lemParityCritPullback}.
\end{proof}
\end{lem}

\subsubsection{Locality of multiplicities}
The multiplicities in the product are the products of multiplicities. More precisely:
\begin{lem}\label{lemExteriourProduct}
 Let $(X,\Lambda)$ and $(Y,\Lambda')$ be two stratified varieties
and $\E=\F,\K$ be a field for simplicity. Then we have 
the following relations between objects of $\P_{\Lambda\times \Lambda'}(X\times Y, \E)$:
\begin{enumerate}
\item $\Delta_{\l} \boxtimes \Delta_{\l'}=\Delta_{\l,\l'}$
\item $IC_{\l} \boxtimes IC_{\l'}=IC_{\l,\l'}$
\item $[\Delta_\l : IC_\mu] \cdot [\Delta_{\l'} : IC_{\mu'}] = [\Delta_{\l,\l' }: IC_{\mu,\mu' }]$
\end{enumerate}
The analogues of these identities also hold if we replace $X,Y$ by $X_0,Y_0$.
\begin{proof}
First of all $\boxtimes$ is exact for the perverse $t$-structure (over a field) by \ref{lemBoXProd}. 
\begin{enumerate}
\item Using exactness properties of our functors, we compute:

\begin{align*}
 \Delta_{\l} \boxtimes \Delta_{\l'} &= (l_{\l!} \K[d_\l]) \boxtimes (l_{\l'!} \E[d_\l']) \\
&=l_{(\l,\l')!} \E[d_\l+d_\l']  \\
&=\Delta_{\l,\l'}\\
\end{align*}


\item This is contained in \ref{lemBoXProd}.
\item Since $\boxtimes$ is exact and maps pairs of $IC$ sheaves to $IC$ sheaves, this follows by tensoring composition series 
of $\Delta_{\l}$ and $\Delta_{\l'}$.
\end{enumerate}
\end{proof}
\end{lem}

\begin{lem}\label{lemLocMult}
 Let $(X,\Lambda)$ be a stratified variety and $U$ be an open subset.
Suppose that $\l,\mu$ are such that 
$$X_\l \cap U \neq \emptyset \neq X_\mu \cap U$$
Let $\E=\F,\K$ be a field. Then the multiplicities $[\Delta_\l:IC_\mu]$ coincide in $\P(U,\E)$ and $\P(X,\E)$. If we replace $X,U$ by
$X_0,U_0$ the analogue identity still holds.
\begin{proof}
This follows by choosing a composition series of $\Delta_l \in \P(X,\E)$, using that restriction to an open subset 
is exact and maps $IC_\nu$ to $IC_\nu$ or $0$ depending on whether $X_\nu \cap U=0$.
\end{proof}
\end{lem}

\begin{lem}\label{lemBalancedProdMult}
Let $(B,\Sigma)$ and $(X,\Lambda)$ be acyclically stratified varieties and $X$ be equipped with a compatible $G$-action.
Let $E\rar B$ be a principal $G$-bundle.
Then the multiplicities $[\Delta_{\sigma,\l}:IC_{\sigma',\l'}]$ coincide in 
$\P_{\Sigma \times \Lambda}(E\underset{G}{\otimes} X,\K)$ and $\P_{\Sigma \times \Lambda}(B\times X,\K)$.
The analogue holds if we replace $X$ by $X_0$ and $E\rar B$ by $E_0 \rar B_0$.
\begin{proof}
 If $E\rar B$ is trivial there is nothing to show. Now by \ref{lemLocMult} the multiplicity $[\Delta_{\sigma,\l}:IC_{\sigma',\l'}]$ 
can be computed after restricting to an open subset which intersects both strata nonempty. 
Hence it suffices to show that for every pair of strata $E_{\sigma}\underset{G}{\otimes} X_{\l}$ and $E_{\sigma'}\underset{G}{\otimes} X_{\l'}$ such that
$$ E_{\sigma'}\underset{G}{\otimes} X_{\l'} \subset \ol{E_{\sigma}\underset{G}{\otimes} X_{\l}}$$
there exists an open subset $U$ of $B$, over which $E$ is trivial and such that $B_\sigma$ and $B_{\sigma'}$ have both non empty 
intersection with $U$. Since $B_{\sigma'} \subset \ol{B_{\sigma}}$ any $U$ which intersects $B_{\sigma'}$ will automatically also intersect $B_{\sigma}$.
Since $B$ can be covered by opens over which $E$ trivializes, such a $U$ must exist. 
\end{proof}
\end{lem}

\subsubsection{Locality of the BGS-condition}

\begin{lem}\label{LemProdBGSisBGS}
 Let $X,Y$ be varieties, which satisfy the BGS-condition. Then $X\times Y$ satisfies the BGS-condition as well.
Furthermore we have
$$wt(X\times Y)=wt(X)\cdot wt(Y)$$
in this case.
\begin{proof}
The six functors commute with $\boxtimes$ and we have $IC_\l \boxtimes IC_\mu=IC_{\l,\mu}$. 
Hence it is easy to check that $IC_{\l,\mu}$ has Tate cohomology sheaves along strata.

The claim about weights follows from the multiplicity formula 
$$[\Delta_\l : IC_\mu] \cdot [\Delta_{\l'} : IC_{\mu'}] = [\Delta_{\l,\l' }: IC_{\mu,\mu' }]$$
in \ref{lemExteriourProduct}.
\end{proof}
\end{lem}

\begin{lem}\label{lemCondBGSLocal}
 Let $X_0$ be a stratified variety. Let $X_0=\bigcup U_{i,0}$ be an open cover.
Then $X_0$ satisfies the BGS-condition if and only if every $U_{i,0}$ satisfies it.
\begin{proof}
If $X_0$ satisfies BGS-condition, then any open subset does as well. Now suppose that
all $U_{i,0}$ satisfy the BGS-condition. Given a sheaf $IC_{\mu,0}$ on $X_0$ and a stratum $X_{\l,0}$, we need to show that 
$$l_\l^* IC_{\mu,0}$$
has Tate cohomology sheaves. We know the cohomology sheaves of $l_\l^* IC_{\mu,0}$ are local systems. Hence they are determined 
by their restrictions to a dense open subset \cite[4.3.2]{BBD}. Since we find an $U_{i,0}$, which contains a dense subset of $X_{\l,0}$ we are done.
\end{proof}
\end{lem}

\begin{lem}\label{lemBalProdBGSCond}
 Let $(B,\Sigma)$ and $(X,\Lambda)$ be acyclically stratified varieties and $X$ be equipped with a compatible $G$-action.
Let $E\rar B$ be a principal $G$-bundle.
If $B$ and $X$ satisfy the BGS-condition, then so does $E\underset{G}{\otimes} X$.
Furthermore we have
$$wt(E\underset{G}{\otimes} X)=wt(B)\cdot wt(X)$$
in this case.\begin{proof}
If $E=B\times G$ we are done by \ref{LemProdBGSisBGS}. By \ref{lemCondBGSLocal} the BGS-condition can be checked locally,
as well as weights \ref{lemBalancedProdMult}.
\end{proof}
\end{lem}

\section{Dg-algebras, bimodules and formality}
Given a $dg$-ring $A$, we denote by $Dg\-A$ the $dg$-category of right dg-modules over $A$. By $Hot\-A$ we denote its homotopy category and 
$Der\-A$ denotes the derived category.
\subsection{Bimodules}
For many purposes the correct notion of a morphism between rings $R$,$S$ is a $R-S$ bimodule. The analogue is true for $dg$-algebras.
Given two $dg$-rings $R,S$ and a $R-S$ $dg$-bimodule $B$, we obtain adjoint functors between their module categories, 
which induce adjoint functors between their derived categories:
$$
\begin{xy}
 \xymatrix{
Der \- R \ar@/^0.5cm/[r]^{ \_ \overset{L}{\underset{R}{\otimes}} B} & \ar@/^0.5cm/[l]^{RHom_S(B,\_)} Der\- S \\
}
\end{xy}
$$

We will notate this situation more concisely as 
$$
\begin{xy}
 \xymatrix{Der \-R \ar[r]^B & Der\-S
}
\end{xy}
$$
We will also abuse language and speak of bimodules instead of $dg$-bimodules.
If our bimodule is perfect over $S$, it restricts to a functor\footnote{The condition for the adjoint to restrict would be $Hom(B,S)$ perfect over 
$R$. We do not care about it.} between the perfect derived categories:
$$
\begin{xy}
 \xymatrix{per \-R \ar[r]^B & per\-S
}
\end{xy}
$$
\begin{facts}\label{FactsBimodFunctors}
\begin{itemize}
 \item 

 Given three $dg$-rings and bimodules between them 
$$
\begin{xy}
 \xymatrix{R \ar@{-}[r]^M & S \ar@{-}[r]^N & T
}
\end{xy}
$$
the diagram 
$$
\begin{xy}
 \xymatrix{Der\-R \ar[r]^M \ar[rd]_{N \overset{L}{\underset{S}{\otimes}} M} & Der\-S \ar[d]^N \\
 & Der \- T
}
\end{xy}
$$
commutes.

\item A map $M\rar N$ of $R-S$ bimodules is a quasi-isomorphism, if and only if the induced natural transformation 
$ (\bullet)\overset{L}{\underset{R}{\otimes}} M \rar (\bullet)\overset{L}{\underset{R}{\otimes}} N$ is an isomorphism.
\item In particular a quasi-isomorphism between dg-rings induces an equivalence of their derived categories.
\end{itemize}
\begin{proof}
 \cite[6.1 and 6.3]{KellerDerDg}  or \cite[8.1.2.4]{LurieHigherAlg} .
\end{proof}

\end{facts}

Let $A$ be a dg-ring and $M$ be a right dg-module over $A$. Then we may form its endomorphisms dg-algebra $End(M)$
and $M$ becomes a $End(M)-A$ bimodule:
$$
\begin{xy}
 \xymatrix{
Der \- End(M)  \ar[rr]^M  && Der\- A \\
}
\end{xy}
$$
\begin{prop}\label{PropGenerator}
 Suppose, that $M$ belongs to the smallest full dg-subcategory of $Dg \-A$
which contains $A$ and is closed under shifts, mapping cones and direct summands.
Then $M$ is restricts to an equivalence as follows:
$$
\begin{xy}
 \xymatrix{Der \- End(M)  \ar[rr]^M && Der\- A\\
 per\-End(M) \ar[rr]^{\cong}   \ar[u]^{\subset} && \ar[u]^{\subset} \langle M \rangle^{\ominus}
}
\end{xy}
$$
Here $\langle M \rangle^{\ominus}$ denotes the thick subcategory generated by $M$. 
\begin{proof}
Indeed by Beilinson's lemma 
we only need to compute, that our functor is fully faithful on $End(M)$:
\begin{align*}
Hom_{Der\- End(M)}(End(M),End(M)[i])&=H^i(End(M))\\
&=Hom_{Hot\-A}(M,M[i])\\
&=Hom_{Der \- A}(M,M[i])
\end{align*}
The last equality follows from our assumption on $M$.
\end{proof}
\end{prop}

\subsection{Formality}
Recall the notion of formality for $dg$-rings.
\begin{defi}
 Let $R$ be a $dg$-ring. We say that $R$ is formal, if there exists a 
chain of $dg$-ring quasi-isomorphisms connecting $R$ and $H(R)$:
$$R \lar R_1 \rar \ldots \lar R_n \rar H(R)$$
\end{defi}
Often we will abuse language and just say that $H(R)$ is formal, leaving $R$ implicit.
Consider the following problem: 
Let $R$ and $S$ be two formal dg-algebras and $M \in R\-Mod\-S$ be a bimodule. Then by abstract nonsense, the lower horizontal of the following 
diagram is again given by a bimodule $N$. What is this bimodule concretely?
$$
\begin{xy}
 \xymatrix{S \-per \ar[d] \ar[r]^M & R\-per \ar[d] \\
H(S) \-per \ar[r]^N & H(R) \-per
}
\end{xy}
$$
A first guess is $N=H(M)$. While this is wrong in general (not every bimodule is formal), it is true if purity arguments are applicable.
\begin{defi}
Let $k$ be a ring. Let $k$-gMod be the category of graded $k$-modules. Let $k$-dggMod be the category of cochain complexes over $k$-gMod.
We will refer to the two gradings as internal (resp.) cohomological grading.
\begin{itemize}
\item A differential graded algebra, $dgg$-algebra for short, is a monoid object in $k$-dggMod.
\item Let $R,S$ be two $dgg$-algebras. A $R-S$ $dgg$-bimodule is a bimodule object over the monoid objects $R,S$.
We will often be sloppy and use the term graded bimodule instead.
\item We call an object of $k$-dggMod pure (of weight $n$), if its $i$-th cohomology is concentrated in internal degree $i+n$ for all $i$.
\end{itemize}
\end{defi}
For $M\in k -dggMod$ pure of weight $0$, we denote by $S(M)$ the subobject, obtained by truncating away the degrees above the diagonal.
$$
S(M)=\left|
\begin{array}{ccc}
0 & 0 & Z^{1,1} \\
    \uparrow & \uparrow & \uparrow \\
0 & Z^{0,0} & M^{1,0} \\
    \uparrow & \uparrow & \uparrow \\
Z^{-1,-1} & M^{0,-1} & M^{1,-1} \\
\end{array}
\right|
$$
More precisely we have
$$S(M)^{ij}:=
\begin{cases}
 0 & \text{ for } i<j \\
 Z^{ij}:=\ker (d|_{M^{ij}}) & \text{ for }i=j\\
 M^{ij} & \text{ for }i>j
\end{cases}
$$
where $i$ is the internal degree and $j$ is the cohomological degree.
Observe that a pure $dgg$-algebra is automatically of weight zero and in this case $S(A)$ is a $dgg$-subalgebra.
\begin{prop}\label{PropDeligneFormality}
\begin{enumerate}

\item Let $A$ be a pure $dgg$-algebra. Then $A$ is formal. More precisely, the obvious maps are quasi-isomorphisms
of dg-algebras:
$$A \linj S(A) \surj H(A)$$ 
\item Let $A,B$ be two pure $dgg$-algebras and $M$ be a pure graded $A-B$ bimodule, which is perfect over $B$. Then the following diagram, 
whose verticals are equivalences, commutes:
 $$
\begin{xy}
 \xymatrix{A\-per  \ar[rr]^{M} && B\-per  \\
S(A)\-per \ar[u] \ar[d] \ar[rr]^{S(M)} && S(B) \-per \ar[u] \ar[d]\\
H(A)\-per  \ar[rr]^{H(M)} && H(B) \-per 
}
\end{xy}
$$
\end{enumerate}
\begin{proof}
\begin{enumerate}
 \item Straightforward, see for example \cite[Prop. 6]{OlafThesis}.
$$
\left|
\begin{array}{ccc}
R^{-1,1} & R^{0,1} & R^{1,1} \\
    \uparrow & \uparrow & \uparrow \\
R^{-1,0} & R^{0,0} & R^{1,0} \\
    \uparrow & \uparrow & \uparrow \\
R^{-1,-1} & R^{0,-1} & R^{1,-1} \\
\end{array}
\right|
\leftarrow
\left|
\begin{array}{ccc}
0 & 0 & Z^{1,1} \\
    \uparrow & \uparrow & \uparrow \\
0 & Z^{0,0} & R^{1,0} \\
    \uparrow & \uparrow & \uparrow \\
Z^{-1,-1} & R^{0,-1} & R^{1,-1} \\
\end{array}
\right|
\rightarrow
\left|
\begin{array}{ccc}
0 & 0 & H^{1,1} \\
    \uparrow & \uparrow & \uparrow \\
0 & H^{0,0} & 0 \\
    \uparrow & \uparrow & \uparrow \\
H^{-1,-1} & 0 & 0 \\
\end{array}
\right|
$$
\item Let us for example verify, that the square 
$$
\begin{xy}
 \xymatrix{A\-per  \ar[rr]^{M} && B\-per  \\
S(A)\-per \ar[u]^A \ar[rr]^{S(M)} && S(B) \-per \ar[u]^B \\
}
\end{xy}
$$
commutes. By \ref{FactsBimodFunctors} this amounts to constructing an isomorphism between 
$S(M)\otimes_{S(B)} B$ and  $A\underset{A}{\otimes} M$ in the derived category of bimodules. This can be done as follows:
$$S(M)\otimes_{S(B)} B \rar M\otimes_{S(B)} B \lar M \otimes_{S(B)} S(B) \rar M \lar A\underset{A}{\otimes} M$$
\end{enumerate}
\end{proof}
\end{prop}

\subsection{Inverse limits of categories}\label{subsecInverseLimits}
In this subsection, we collect some technical properties of inverse limits of categories.
\begin{lem}\label{LemIdemLimit}
Let 
$$
\begin{xy}
 \xymatrix{\cal C_0 & \ar[l]_{F_1} \cal C_1 & \ar[l]_{F_2} \cal C_2 & \ar[l]_{F_3} \ldots
}
\end{xy}
 $$
be a sequence of additive categories and additive functors. Then $\varprojlim (\cal C_i)$ is additive.
\begin{enumerate}
 \item If the $\cal C_i$ are idempotent complete, then so is  $\varprojlim (\cal C_i)$.
\item Idempotent completion commutes with $\varprojlim$. More precisely
the canonical functor $idem(\varprojlim(\cal C_i)) \rar \varprojlim ( idem(\cal C_i))$ is an equivalence.
\end{enumerate}
\begin{proof}
 Straightforward.
\end{proof}
\end{lem}

\begin{prop}\label{PropPerStetig}
Given a sequence 
$$
\begin{xy}
 \xymatrix{H_0 & \ar[l]_{\psi_1} H_1 & \ar[l]_{\psi_2} H_2 & \ar[l]_{\psi_3} \ldots
}
\end{xy}
 $$
of non-negatively graded algebras, considered as dg-algebras. Assume that 
$\psi_i$ is an isomorphism below degree $i$ and that the degree $0$ part of each algebra is 
a finite product of copies of $\E$. 
Then the canonical functor
\begin{equation}\label{eq1PropPerStetig}
per\-(\varprojlim H_i) \rar \varprojlim (per\-H_i)
\end{equation}
gives a triangulated equivalence.
\begin{proof}
In \cite[86]{OlafThesis} it is proven that 
\begin{equation}\label{eq2PropPerStetig}
\langle \varprojlim H_i \rangle \rar \varprojlim \langle H_i \rangle
\end{equation}
is an equivalence, where $\langle A \rangle \subset Der\-A$ denotes the smallest triangulated subcategory 
containing $A$ as usual.
Since \ref{eq1PropPerStetig} is the idempotent completion of \ref{eq2PropPerStetig} we are done.
\end{proof}
\end{prop}

\section{Limiting step}
Recall that we want to understand $D^b_{G,\Lambda}(X,\E)$ by using the formula 
$$\varprojlim D^b_{\Lambda}(\ol X_n,\E) \cong D^b_{G,\Lambda}(X,\E)$$
In order to do so, we need to describe the transition functors 
$$i^*:D^b_{\Lambda}(\ol X_{n+1},\E)\rar D^b_{\Lambda}(\ol X_n,\E)$$
more explicitly.

Let $R$ be a ring equipped with an endomorphism $\phi$. We denote by $Mod\-(R,\phi)$ the category whose objects 
consist of pairs $(M,\phi)$, where $M$ is an right module over $R$ and $\phi:M\rar M$ is an additive map, which satisfies
$\phi(mr)=\phi(m)\phi(r)$. Typical examples will be obtained as follows: $R=End(P)$ and $M=Hom(P,\cal M)$, 
for suitable perverse sheaves $P_0,\cal M_0$. Both $R$ and $M$ are equipped with the frobenius action in this case.

\begin{obs}\label{ObsResProj}
 Let $i:X\inj X'$ be an inclusion of acyclically stratified varieties such that $X'$ is $IC^\O$-parity. 
Let $P=\bigoplus P_\l$ and $P'=\bigoplus P'_{\l'}$ be our usual projective generators, as
obtained in \cite{janGrass}  . 
Recall that their construction went by starting with projectives
on small varieties and extending them to bigger varieties. In particular we have $P=i^* P'$ by construction. 
Moreover if $i$ is defined over $\F_q$, for any
lift $P_0$, we can choose $P'_0$ such that $P_0=i^* P'_0$. 
In particular, we obtain a map $End(P')\rar End(P)$, which is compatible with the Frobenius action.
\end{obs}
This observation allows to translate the functors $i_*$ and $i^*$ nicely:
\begin{lem}\label{LemProjectivesAndFunctors}
Let $X\inj X'$ be an inclusion of acyclically stratified varieties. Assume that $X'$ is $IC^\O$-parity.
\begin{itemize}
 \item Then the following diagram of adjunctions commutes:
\footnote{More precisely, by ``commuting diagram of adjunctions'' we only
mean that there are natural isomorphisms
\begin{align*}
 |Hom(P,\_)| & \cong Hom(P',i_* \_) \\
 Hom(P,Li^* \_) & \cong End(P) \overset{L}{\underset{End(P')}{\otimes}} Hom(P',i_* \_)
\end{align*}
We do not claim that the adjunctions are equivalences.
} 
$$
\begin{xy}
 \xymatrix{D^b(\P_{\Lambda'}(X',\E)) \ar[dd]_{Hom(P',\_)} \ar@/^0.5cm/[r]^{Li^*} & \ar@/^0.5cm/[l]^{i_*}   D^b(\P_{\Lambda}(X,\E)) \ar[dd]^{Hom(P,\_)} \\
& \\
D^b(mod \-End(P')) \ar@/^0.5cm/[r]^{End(P) \overset{L}{\underset{End(P')}{\otimes}} \_} & \ar@/^0.5cm/[l]^{|.|} D^b(mod \- End(P))
}
\end{xy}
$$
\item Assume that the inclusion $X\inj X'$ comes from $X_0\inj X'_0$. Then the following diagram of adjunctions commutes:
$$
\begin{xy}
 \xymatrix{D^b(\P_{\Lambda'}(X'_0,\E)) \ar[dd]_{Hom(P',\_)} \ar@/^0.5cm/[r]^{Li^*} & \ar@/^0.5cm/[l]^{i_*}   D^b(\P_{\Lambda}(X_0,\E)) \ar[dd]^{Hom(P,\_)} \\
& \\
D^b(mod \-(End(P'),Fr)) \ar@/^0.5cm/[r]^{End(P) \overset{L}{\underset{End(P')}{\otimes}} \_} & \ar@/^0.5cm/[l]^{|.|} D^b(mod \- (End(P),Fr))
}
\end{xy}
$$
\end{itemize}
\begin{proof}
Indeed we compute:
$$
 |Hom(P,\_)| =Hom(i^* P', \_) = Hom(P',i_* \_) 	    
$$
and these isomorphisms are compatible with the module structures and Frobenius action.
The rest follows from adjunction properties.
\end{proof}
\end{lem}

Let us consider triples $(X,\Lambda,\wt \Lambda)$, where $(X,\Lambda)$ is 
a stratification and $(X,\wt \Lambda)$ is a refinement into an acyclic stratification. 

Given such a triple and $m\in \Z$ we denote by $IC_\Lambda:=IC_\Lambda^\E:=\bigoplus_{\l \in \Lambda} IC_\l^\E[m]$
the direct sum of $IC$-sheaves, normalized to perverse degree $-m$.

Now let $(X,\Lambda,\wt \Lambda)$ and $(X',\Lambda,\wt \Lambda')$ be two such triples (same $\Lambda$!).
We will sometimes use notations $\cal F$ and $\cal F'$ to distinguish sheaves on $X$ and $X'$.
For example $IC'_\l$ denotes an $IC$-sheaf on $X'$, while $IC_\l$ lives on $X$ etc.

Let $i:X\inj X'$ be a closed inclusion, such that 
\begin{itemize}
 \item $i:(X,\wt \Lambda) \rar (X',\wt\Lambda')$ is a map of acyclically stratified varieties.
\item For each $\l \in \Lambda$ restriction gives a normally smooth inclusion of codimension $c$ between
strata closures:
$$\ol{X_\l} \inj \ol{X'_\l}$$
\end{itemize}
Observe that under these conditions we have $i^*IC'_\l =IC_\l[c]$ for all $\l \in \Lambda$
and hence 
$$i^*IC'_\Lambda=IC_\Lambda$$
by \ref{PropICclosed}.
Here and in future our normalizations are $m:=0$ and $m':=-c$. In other words
$IC_\Lambda$ and $IC'_\Lambda[c]$ are perverse.
\begin{thm}\label{thmFormalityInclusion}
Let $X \inj X'$ be as above and assume, that everything\footnote{$X_0\inj X_0'$ and all stratifications.} 
is defined over $\F_q$ such that $(X',\wt \Lambda')$ satisfies the BGS-condition.
Suppose in addition that $IC^\O_{\l'}$ is parity for all $\l'\in \wt \Lambda'$ and that $wt(X')$ is separated.
Then we have a commutative diagram:
$$
\begin{xy}
 \xymatrix{D^b_{\Lambda}(X',\E)  \ar[rr]^{i^*}  && D^b_{\Lambda}(X,\E)  \\
\langle IC'_{\Lambda} \rangle \ar[rr]^{i^*} \ar[u]^{\subset} \ar[d]^{\cong} && \langle IC_{\Lambda} \rangle \ar[d]^{\cong} \ar[u]^{\subset}\\
per \-Ext^\bullet(IC'_{\Lambda}) \ar[rr]^{Ext^\bullet(IC_\Lambda)} && per \-Ext^\bullet(IC_{\Lambda}) \\
}
\end{xy}
$$
Here the lower verticals are equivalences to be constructed and the bottom horizontal is extension of scalars (in bimodule notation).
\begin{proof}
This proof is essentially \cite[Chapter 3]{OlafThesis} . We reproduce here a variant of it.

\begin{lem}\label{LemHugeDiag}
 The following diagram commutes:
$$
\begin{xy}
 \xymatrix{
D^b_{\wt \Lambda'}(X',\E)  \ar[rrrr]^{i^*}  &&&& D^b_{\wt \Lambda}(X,\E) \\
D^b(\P_{\wt \Lambda'}(X',\E)) \ar[d]_{Hom(P',\_)} \ar[u]^{real} \ar[rrrr]^{Li^*}  &&&& \ar[d]^{Hom(P,\_)} \ar[u]_{real} D^b(\P_{\wt \Lambda}(X,\E)) \\
per \-A'  \ar[rrrr]^{A}&&&&  per \- A \\
per \-End(L'^\bullet) \ar[u]^{L'^\bullet} \ar[rrrr]^{Hom(L^\bullet, L'^\bullet \underset{A'}{\otimes} A)}  &&&& per \-End(L^\bullet) \ar[u]_{L^\bullet}\\
per \-SEnd(L'^\bullet) \ar[d]_{HEnd(L'^\bullet)} \ar[u]^{End(L'^\bullet)} \ar[rrrr]^{SHom(L^\bullet, L'^\bullet \underset{A'}{\otimes} A)}  &&&& per \-SEnd(L^\bullet) \ar[d]^{HEnd(L^\bullet)} \ar[u]_{End(L^\bullet)}\\
per \-HEnd(L'^\bullet)  \ar[rrrr]^{HHom(L^\bullet, L'^\bullet \underset{A'}{\otimes} A)}  &&&& per \-HEnd(L^\bullet) \\
per \-Ext^\bullet(L') \ar@{=}[u]  \ar[rrrr]^{Ext^\bullet(L,L'\underset{A'}{\otimes} A)} &&&& \ar@{=}[u] per \-Ext^\bullet(L) \\
per \-Ext^\bullet(L') \ar@{=}[u]  \ar[rrrr]^{Ext^\bullet(L,L)} &&&& \ar@{=}[u] per \-Ext^\bullet(L) \\
}
\end{xy}
$$
Here we used the following notation:
\begin{itemize}
\item $P_0=i^*P'_0$ and $P'_0$ are quasi-projective generators of $\P_{\Lambda}(X,\E)$ and $\P_{\Lambda'}(X',\E)$ as obtained in \cite{janGrass}
.
\item $A:=End(P)$ and $A':=End(P')$.
\item $L:=Hom(P,IC_{\Lambda})$ and $L'=Hom(P',IC_{\Lambda}')$ correspond to $IC_\Lambda$ and $IC_\Lambda'$.
\item $L^\bullet$ and $L'^\bullet$ denote say finite resolutions by graded finitely generated projectives of $L,L'$.
\end{itemize}

\begin{proof}[Proof of \ref{LemHugeDiag}]
We need to check, that all squares commute. We proceed from top to bottom.
For the first square commutativity follows from \cite[A.7.1.]{DbPerv}.

The second square is covered in \ref{LemProjectivesAndFunctors}. Keep in mind, 
that $\P_{\wt \Lambda}(X,\E)$ has finite cohomological dimension 
, so $per\-End(P)=D^b(mod \-End(P))$.

All other squares can be interpreted as bimodules between perfect derived categories.
First of all we should convince ourselves, that 
that every bimodule in the diagram is perfect over its target $dg$-algebra. This is not hard to check,
most bimodules are quasi-isomorphic to their target dg-algebra anyway.

For the third square, we need to show that the evaluation map 
$$Hom(L^\bullet, L'^\bullet \underset{A'}{\otimes} A) \otimes_{End(L^\bullet)} L^\bullet \rar L'^\bullet \underset{A'}{\otimes} A$$
is a quasi-isomorphism. By \ref{PropICclosed} we have $L\cong L' \underset{A'}{\otimes}^L A$, which implies that 
$L'^\bullet \underset{A'}{\otimes} A $ and  $L^\bullet$ are both projective resolutions of the same object. Hence 
there is a homotopy equivalence $\phi:L'^\bullet \underset{A'}{\otimes} A \rar L^\bullet $.
It gives rise to a commutative diagram:
$$
\begin{xy}
 \xymatrix{Hom(L^\bullet, L'^\bullet \underset{A'}{\otimes} A) \otimes_{End(L^\bullet)} L^\bullet \ar[d]_{(\phi \circ \_)\otimes id} \ar[r] & L'^\bullet \underset{A'}{\otimes} A \ar[d]^{\phi} \\
Hom(L^\bullet, L^\bullet) \otimes_{End(L^\bullet)} L^\bullet \ar[r] & L^\bullet
}
\end{xy}
$$
Now both verticals are homotopy equivalences, since $\phi$ is and the bottom horizontal is an isomorphism anyway. It follows that 
upper horizontal is a homotopy equivalence.

The fourth and fifth square are covered by \ref{PropDeligneFormality}.

The sixth square is straightforward.

For the seventh square one uses again $L'\underset{A'}{\otimes} A \cong L$ by \ref{PropICclosed}.
\end{proof}
\end{lem}

Now the horizontals of all squares but the third are equivalences by \ref{PropGenerator},\ref{PropDeligneFormality} and \cite[2.3.4.]{RSW}.
By \ref{PropGenerator} the horizontals of the third square still induce 
equivalence between the categories corresponding to $\langle IC_\Lambda' \rangle $ resp. $\langle IC_\Lambda  \rangle$.
\end{proof}
\end{thm}

\begin{rem}
 Note, that the choices of projective resolutions $L^\bullet,L'^\bullet$ in the theorem were independent of each other. 
Moreover we already remarked in \ref{ObsResProj}, that for any lift $P_0$, we find $P'_0$ such that $i^*P'_0=P$.
This means that even for an infinite chain of closed embeddings
$$\ldots \linj X_0''\linj  X_0'\linj  X_0$$
we still have a commutative diagram:
$$
\begin{xy}
 \xymatrix{
\ldots 	\ar[r]^{i^*}  &			    D^b_{\Lambda}(X'',\E)  \ar[r]^{i^*}  & D^b_{\Lambda}(X',\E) 	\ar[r]^{i^*}  					 & D^b_{\Lambda}(X,\E)  \\
\ldots 	\ar[r]^{i^*}  &	\langle IC''_{\Lambda} \rangle \ar[r]^{i^*} \ar[u]^{\subset} \ar[d]^{\cong} & \langle IC'_{\Lambda} \rangle \ar[r]^{i^*} \ar[u]^{\subset} \ar[d]^{\cong}	 & \langle IC_{\Lambda} \rangle \ar[d]^{\cong} \ar[u]^{\subset}\\
\ldots 	\ar[r]^{i^*}  &						  per \-Ext^\bullet(IC''_{\Lambda}) \ar[r] & per \-Ext^\bullet(IC'_{\Lambda}) \ar[r]						 & per \-Ext^\bullet(IC_{\Lambda}) \\
}
\end{xy}
$$
Here we assume all $wt(X),wt(X'),wt(X''),\ldots$ to be separated and all varieties to satisfy 
$IC^\O_\l$-parity ($\l \in \wt \Lambda_n$) and the BGS-condition.
\end{rem}

\section{Equivariant formality}

Let $(X_n,X,G)$ be an approximation. 
Suppose that everything is defined over $\F_q$:
I.e. assume that we are given a diagram
$$
\begin{xy}
\xymatrix{
 & & X_0  \\
X_{0,0} \ar[r] \ar[urr] \ar[d] & X_{1,0} \ar[r] \ar[d] \ar[ur] & X_{2,0} \ar[r] \ar[d] \ar[u] &\ldots \\
\ol X_{0,0} \ar[r] & \ol X_{1,0} \ar[r] & \ol X_{2,0} \ar[r] & \ldots \\
}
\end{xy}
$$
decompositions $(\ol X_{n,0},\Lambda,\wt \Lambda)$ a group $G_0$ etc.
\begin{defi}\label{defBGSequiv}
In the above situation, we say that $(X_{n,0},X_0,G_0)$ satisfies the BGS-condition, if
all $(\ol X_n,\wt \Lambda_n)$ satisfy the BGS-condition and in addition
$$wt(X,G):=\bigcup wt(X_{n,0})$$
is a finite set.

We say that $(X_{n},X,G)$ satisfies $IC^\O$-parity, if for all $n$ and all $\l \in \wt \Lambda_n$ 
the sheaf $IC^\O_\l$ is parity.
\end{defi}
We will abuse notation and just say that ``$(X_n,X,G)$ satisfies the BGS-condition'' or ``$(X_n,X,G)$ is BGS''.
The BGS-condition and $IC^\O$-parity are both preserved under the usual constructors:
\begin{thm}\label{thmBGSApproxRules}
\begin{enumerate}
\item\label{itemProd} Let $(X_n,X,G)$ and $(Y_n,Y,H)$ be two approximations,
which satisfy the BGS-condition (resp. $IC^\O$-parity). Then 
$$(X_n\times Y_n,X\times Y,G\times H)$$
satisfies the BGS-condition (resp. $IC^\O$-parity) as well.
In this case we have:
$$wt(X\times Y, G\times H)=wt(X,G) \cdot wt(Y,H)$$
\item\label{itemSplit} Let $N \inj P \surj L$ be a short exact sequence of algebraic groups, for which $P\surj L$ is split. 
Suppose that $N$ is acyclic. If $(E_n,pt,L)$ is BGS (resp. $IC^\O$-parity), then $(P\underset{L}{\otimes} E_n,pt,P)$ BGS  (resp. $IC^\O$-parity) as well.
In this case we have:
$$wt(pt,P)=wt(pt,L)$$
\item\label{itemBalanced} Let $X$ be a variety with a $G$-action and compatible acyclic stratification and $(E_n,pt,G)$ be an approximation. 
Suppose that $X$ and $(E_n,pt,G)$ both satisfy the BGS-condition (resp. $IC^\O$-parity).
Then $(E_n \times X,X,G)$ satisfies the BGS-condition  (resp. $IC^\O$-parity) as well. 
In this case we have:
$$wt(X,G)=wt(X)\cdot wt(pt,G)$$
\end{enumerate}
\begin{proof}
It is straightforward to assemble the proof from the following ingredients:
 \begin{enumerate}
\item \ref{thmApproxConstr}, \ref{LemProdBGSisBGS}
\item \ref{thmApproxConstr}
\item \ref{thmApproxConstr}, \ref{lemBalProdBGSCond}
 \end{enumerate}
\end{proof}
\end{thm}

The theorem allows us to construct BGS-approximations for any partial flag variety and compute weights in many cases:
\begin{ex}
\begin{itemize}
\item Let $G=\mathbb G_m$, then $\mathbb P^N$ gives a BGS and $IC^\O$-parity approximation of $(pt,G)$. Hence we get
$$wt(pt,\mathbb G_m)=\{1,\q\}$$
By taking products we also get
$wt(pt,T)=\{1,\q,\ldots \q^k\}$
for $T=\mathbb G_m^k$ a torus.
\item Let $G=GL_k$. Then $Gr(k,N)$ gives a BGS and $IC^\O$-parity approximation of $(pt,G)$. By \cite{janGrass}
we get
$$wt(pt,GL_k)=\{1,\q,\ldots ,\q^k \}$$
By taking products and forming a split extension, we get 
$$wt(pt,P)=\{1,\q,\ldots, \q^n\}$$
where $B \subseteq P\subseteq GL_n$ is a parabolic.
\item Let $T \subset G$ be a maximal torus inside a connected solvable group.
We know that a BGS and $IC^\O$-parity approximation of $(pt,T)$ exists with $wt(pt,T)=\{1,\q,\ldots \q^{rk(T)}\}$.
By forming a split extension, we also get a $IC^\O$-parity BGS approximation $(pt,G)$. It satisfies 
$$wt(pt,G)=\{1,\q,\ldots, \q^{rk(T)}\}$$
For example we get approximations for Borel subgroups this way.
\item Let $B\subset P \subset G$ be a Borel inside a parabolic inside a connected reductive group.
Let $X:=G/P$. Applying the third point of \ref{thmBGSApproxRules} we find a BGS-approximation of $(X,B)$. 
If $P=B$, we have 
$$wt(G/B,B)=\{1,\q, \ldots, \q^{\dim G/B+rk(T)} \}$$
by \cite{janGrass} 
.
This approximation is $IC^\O$-parity, if and only if $X$ is.
\item Let $B\subset GL_n$ be the Borel of upper triangular matrices and $X:=Gr(k,n)$
be the Grassmannian equipped with the usual $B$-action. Again \ref{thmBGSApproxRules} gives us a BGS and $IC^\O$-parity approximation of $(X,B)$.
Furthermore we have 
$$wt(X,B)=\{1,\q,\ldots, \q^{n+min(k,n-k)}\}$$
\end{itemize}
\end{ex}

\begin{thm}\label{MainThmEquivFormality}
 Let $(X,\Lambda)$ be an acyclically stratified variety with compatible $G$-action.
Suppose that there exists an approximation $(X_n,X,G)$ which is BGS and $IC^\O$-parity.
Suppose that $wt(X,G)$ is separated. 
Then there exists an equivalence of categories:
$$D_{G,\Lambda}^b(X,\E)\cong per\-Ext^\bullet(IC) $$
\begin{proof}
 We have
\begin{align*}
 D_{G,\Lambda}^b(X,\E)& \cong \varprojlim D_{G,\Lambda}^b(\ol X_n,\E) && \ref{thmOlafApprox}\\
		      & \cong \varprojlim per\-Ext^\bullet(IC) && \ref{thmFormalityInclusion}\\
		      & \cong  per\- (\varprojlim Ext^\bullet(IC)) && \ref{PropPerStetig}\\
		      & \cong per\-Ext^\bullet(IC) && \ref{thmOlafApprox}
\end{align*}
\end{proof}
\end{thm}

\begin{cor}
Let $(X,\Lambda)$ be an acyclically stratified variety with compatible $G$-action
such that the $IC^\O$-sheaves are parity. Let $(E_n,pt,G)$ be an approximation. 
Suppose that $X$ and $(E_n,pt,G)$ both satisfy the BGS-condition. Suppose that 
$$wt(X,G)=wt(X)\cdot wt(pt,G)$$
is separated. 
Then there exists an equivalence of categories:
$$D_{G,\Lambda}^b(X,\E)\cong per\-Ext^\bullet(IC_\Lambda,IC_\Lambda) $$
\begin{proof}
This follows from \ref{MainThmEquivFormality} and \ref{thmBGSApproxRules}.
\end{proof}
\end{cor}

\begin{cor}
Let $B\subset P \subset G$ be a Borel inside a parabolic inside a connected, reductive group over $\ol \F_q$. Then for $l>>0$ there exists an 
equivalence of categories:
$$D_{B}^b(G/P,\F)\cong per\-Ext^\bullet(IC)$$
\end{cor}

\begin{cor}\label{corGrass}
 Suppose that $\{1,\q,\ldots ,\q^{n+min(n-k,k)} \}$ is separated. Then there is an equivalence of categories:
$$D^b_B(Gr(k,n),\E)\cong per\-Ext^\bullet(IC,IC)$$
\end{cor}

\subsection{Passage from $X_{\ol \F_q}$ to $X_\C$}
So far, we established equivariant formality for varieties over $\ol \F_q$. We will now briefly explain how to extend these 
results to varieties over $\C$. 
\begin{thm}
 Let $X_{\C}=Gr(n,k)_{\C}$ be the Grassmannian of $k$-planes inside $\C^n$,
equipped with the usual action by the set of upper triangular invertible matrices $B_\C$. If 
$$l>n+min(k,n-k)+1$$
 then there exists an equivalence of categories
$$D^b_{B_{\C}}(X_\C,\E) \cong per \-Ext^\bullet(IC)$$
\begin{proof}
First of all we may replace the analytic topology by the étale topology on complex varieties \cite{BBD}[6.1.2].

Let $V \subset \C$ be a strictly Henselian local ring with residue field $\ol\F_p$. 
Objects over $V$ are denoted by $X_V$ and their base changes to $\C,\ol \F_q$ are denoted by $X_\C,X$.

Then $B,Gr(n,k),E(n,k)$ are defined over $V$ in a way that satisfies 
the conditions in \cite[6.1.8]{BBD} (Indeed everything is even defined over $\Z$). Furthermore all operations \ref{thmApproxConstr} used to construct
the relevant approximation make sense over $V$. Note in particular that the strata of $(X_{n,V}, \wt \Lambda_n)$ 
are still acyclic over $V$, since they are of $\mathbb A_n$ bundles over $\mathbb A_{n,V}$.
As in \cite[7.1.4]{RSW} we now obtain vertical pullback equivalences, which fit into a commutative diagram
and preserve $IC_{\Lambda}$-sheaves:
$$
\begin{xy}
 \xymatrix{
\ldots \ar[r]&	D^b_{\wt \Lambda}(X_{3,\C},\E) 	\ar[r]^{i^*}  &	D^b_{\wt \Lambda}(X_{2,\C},\E) 	\ar[r]^{i^*} &	D^b_{\wt \Lambda}(X_{1,\C},\E)\\				
\ldots \ar[r]&	D^b_{\wt \Lambda}(X_{3,V},\E) \ar[u] \ar[d] \ar[r]^{i^*} & D^b_{\wt \Lambda}(X_{2,V},\E) \ar[u] \ar[d] \ar[r]^{i^*} & D^b_{\wt \Lambda}(X_{1,V},\E) \ar[u] \ar[d]\\
\ldots \ar[r]&	D^b_{\wt \Lambda}(X_3,\E) 	\ar[r]^{i^*} &  D^b_{\wt \Lambda}(X_2,\E) 	\ar[r]^{i^*} & D^b_{\wt \Lambda}(X_1,\E) 
}
\end{xy}
$$
Assuming that $l>wr(X,B)=n+min(k,n-k)+1$, we may invoke Dirichlets theorem to choose our prime $p$ such that $wt(X,B)$ is separated with respect to $l$.
Since all $IC^\O$ are parity for the Grassmannian, we get the desired result:
\begin{align*}
D_{B_{\C}}^b(X_\C,\E) & \cong D_B^b(X,\E) \\
 &\cong per\-Ext_{D_B^b(X,\E)}(IC,IC) \\
&\cong per\-Ext_{D_{B_{\C}}^b(X_\C,\E)}(IC,IC)
\end{align*}
\end{proof}
\end{thm}
Exploiting that our groups and partial flag varieties are defined over the integers \cite[II.1.1.9]{JantzenAlgGrp}, 
the same proof gives us the following:
\begin{thm}
Let $G_\C$ be a connected reductive complex algebraic group and $X_\C=G_\C/P_\C$ be a partial flag variety. 
Assume that all $IC^\O$-sheaves are parity and that $l>wr(X,B)$. Then there is an equivalence of categories:
$$D^b_{B_\C}(X_\C,\E)\cong per \-Ext^\bullet(IC)$$
\end{thm}

\begin{tabular}{l | l | l | l}
Action of $G$ on $X$ & $wt(X)$ & $wt(X,G)$ & $IC^\O$-parity \\
&&&\\
\hline
&&&\\
$B$ action on $G/B$ for & $\{1,\q, \ldots, \q^{\dim G/B} \}$ & $\{1,\q, \ldots, \q^{\dim G/B+rk(T)} \}$ & ?? \\
$G$ connected reductive &&&\\
\hline
&&&\\
$P$ action on $pt$ for & $\{1\}$ & $\{1,\q, \ldots, \q^{n} \}$ & True for any $l$ \\
$P\subset GL_n$ parabolic &&&\\
\hline
&&&\\
$B$ action on $Gr(k,n)$ & $\{1,\q, \ldots, \q^{min(k,n-k)} \}$ & $\{1,\q, \ldots, \q^{min(k,n-k)+n} \}$ & True for any $l$ \\
$B\subset GL_n$ Borel &&&\\
\hline
&&&\\
$B$ action on $G/B$ &&&\\
for $G$ semisimple of type &&&\\
\hline
&&&\\
$A_k$ for $k\leq 6$ & $\{1,\q, \ldots, \q^{\frac{k(k+1)}{2} } \}$ & $\{1,\q, \ldots, \q^{\frac{k(k+3)}{2} } \}$ & True for any $l$ \\
&&&\\
\hline
&&&\\
$A_7$ & $\{1,\q, \ldots, \q^{56} \}$ & $\{1,\q, \ldots, \q^{63} \}$ & True iff $l \neq 2$ \\
&&&\\
\hline
&&&\\
$B_2$ & $\{1,\q, \ldots, \q^{4} \}$ & $\{1,\q, \ldots, \q^{6} \}$ & True iff $l \neq 2$ \\
&&&\\
\hline
&&&\\
$D_4$ & $\{1,\q, \ldots, \q^{12} \}$ & $\{1,\q, \ldots, \q^{16} \}$ & True iff $l \neq 2$ \\
&&&\\
\hline
&&&\\
$G_2$ & $\{1,\q, \ldots, \q^{6} \}$ & $\{1,\q, \ldots, \q^{8} \}$ & True for any $l$ \\
&&&\\
\hline
\end{tabular}

The table summarizes the authors knowledge about BGS-approximations, weights and $IC^\O$-parity on partial flag varieties.
The results about $IC^\O$-parity for $G/B$ are taken from \cite{GeordieModInt}.

\bibliographystyle{alpha}

\bibliography{referenzen}

\end{document}